\documentclass[hidelinks,onefignum,onetabnum]{siamart251104}

\usepackage[utf8]{inputenc}
\usepackage[T1]{fontenc}
\usepackage{amsmath,amssymb,amsfonts}
\usepackage{mathrsfs}
\usepackage{mathtools}
\usepackage{enumitem}

\newsiamthm{assumption}{Assumption}
\newsiamthm{problem}{Problem}
\newsiamremark{remark}{Remark}

\crefalias{assumption}{assumption}

\crefname{equation}{}{}
\crefname{assumption}{Assumption}{Assumptions}
\Crefname{assumption}{Assumption}{Assumptions}
\crefname{lemma}{Lemma}{Lemmas}
\crefname{corollary}{Corollary}{Corollaries}
\crefname{proposition}{Proposition}{Propositions}
\crefname{definition}{Definition}{Definitions}
\crefname{theorem}{Theorem}{Theorems}

\headers{Viscosity Solutions of Stochastic HJB Equations}{D. Liang and Q. Meng}

\title{Viscosity Solutions of Stochastic Hamilton--Jacobi--Bellman Equations with Jumps\thanks{Submitted to the editors \today.
		\funding{This work was supported by the National Natural Science Foundation of China (No. 12271158).}}}

\author{Dunxiang Liang\thanks{Department of Mathematics, Huzhou Normal University, Zhejiang 313000, China (\email{lianglion3895@gmail.com}).}
	\and Qingxin Meng\thanks{Corresponding author. Department of Mathematics, Huzhou Normal University, Zhejiang 313000, China (\email{mqx@zjhu.edu.cn}).}}

\begin{document}
	\maketitle
	
	\begin{abstract}
		This paper studies the stochastic optimal control of jump-diffusion processes and the associated fully nonlinear backward stochastic Hamilton--Jacobi--Bellman (BSHJB) equations. We establish the dynamic programming principle (DPP) via backward semigroups to characterize the value function. To handle non-local integro-differential operators and polynomial growth, we introduce a stochastic viscosity solution framework based on semimartingale test functions and global tangency conditions. Existence is proved using the measurable selection theorem and the generalized It\^o--Kunita formula. Finally, under a super-parabolicity condition, we establish a weak comparison principle and prove global uniqueness via localized bounding envelopes and backward induction.
	\end{abstract}
	
	
	
	\section{Introduction}
	
	The optimal control of stochastic systems in random environments has led to the extensive development of backward stochastic differential equations (BSDEs) \cite{PardouxPeng90} and the comprehensive theory of forward-backward stochastic differential equations (FBSDEs) \cite{MaYong99Book, MaYong07Book}. A central object arising from such studies is the Hamilton--Jacobi--Bellman (HJB) equation, which characterizes the value function of a stochastic optimal control problem. In the classical Markovian framework with deterministic coefficients, the HJB equation is a deterministic second-order PDE whose well-posedness under smoothness conditions is well understood. When the coefficients are random, however, the associated HJB equation becomes a backward stochastic partial differential equation (BSPDE) whose solution is a random field. The additional presence of non-local integral operators from jump processes and polynomial growth in the coefficients introduces profound analytical challenges that render both classical smooth solutions and standard Sobolev-space techniques inapplicable.
	
	The theory of stochastic HJB equations with random coefficients has been developed along two complementary lines: classical/weak solutions in Sobolev spaces, and viscosity solutions. Along the first line, Peng \cite{Peng1992} studied the stochastic HJB equation and established the existence and uniqueness of a Sobolev weak solution under the assumption that the diffusion coefficient does not depend on the control. Subsequently, Peng \cite{Peng1997} introduced the backward stochastic semigroup method, which provides a powerful tool for characterizing the value function via dynamic programming in stochastic optimal control problems. More recently, Meng et al. \cite{MengDongShenTang2023} extended this Sobolev framework to jump-diffusion systems with random coefficients, proving that the value function is a classical solution of the associated stochastic HJB equation under suitable regularity assumptions. While these results provide rigorous foundations, they inherently require sufficient Sobolev regularity of the solution, a condition that is generally not satisfied in fully nonlinear settings with non-smooth value functions.
	
	The viscosity solution approach, pioneered by Crandall and Lions \cite{CrandallLions83}, offers an alternative that accommodates non-smooth solutions. For deterministic jump-diffusion systems, Li and Peng \cite{LiPeng2009} extended the viscosity solution theory to HJB equations with non-local integral operators, proving that the value function is a viscosity solution of the associated nonlocal HJB integro-partial differential equation. In the stochastic setting with random coefficients, Buckdahn and Ma \cite{BuckdahnMa01I, BuckdahnMa01II, BuckdahnMa02} developed the concept of stochastic viscosity solutions for fully nonlinear SPDEs via a Doss--Sussmann-type transformation, later unified with rough path theory \cite{BuckdahnKellerMaZhang20, BuckdahnMa07}. Building on this foundation, Qiu \cite{Qiu2018} introduced an alternative and more direct framework based on semimartingale test functions that directly incorporate the stochastic dynamics of the controlled system, proving that the value function is the maximal viscosity solution of the associated stochastic HJB equation, with uniqueness established under a super-parabolicity condition. Qiu and Wei \cite{QiuWei2019} further proved uniqueness under standard Lipschitz assumptions for stochastic HJ equations, and Qiu and Zhang \cite{QiuZhang2023} extended the theory to zero-sum two-player stochastic differential games. However, these stochastic viscosity solution frameworks are developed exclusively for continuous diffusions and do not account for jumps.
	
	In parallel, the viscosity solution theory was extended to path-dependent settings by Ekren, Touzi, and Zhang \cite{EkrenTouziZhang2016a, EkrenTouziZhang2016b, EkrenZhang2016}, who established a rigorous analytical foundation for fully nonlinear path-dependent PDEs. Within this framework, Tang and Zhang \cite{TangZhang2013} characterized the value functional as the unique viscosity solution of the associated path-dependent Bellman equation. Zhou \cite{Zhou2023} extended this theory to second-order path-dependent HJB equations, highlighting applications to backward stochastic HJB equations, and Zhou, Touzi, and Zhang \cite{ZhouTouziZhang2024} further extended the theory to mean field control with common noise. These works provide conceptual tools for managing randomness in viscosity solutions, albeit within a path-dependent infinite-dimensional setting rather than the finite-dimensional Markovian framework of the present paper.
	
	A critical gap emerges from the above review: the existing literature covers three dimensions of difficulty---random coefficients, non-local jump operators, and viscosity solutions---but no single work addresses all three simultaneously. More specifically, (i) the classical/weak solution approach \cite{Peng1992, MengDongShenTang2023} handles jumps and random coefficients but requires Sobolev regularity; (ii) the deterministic viscosity solution theory \cite{LiPeng2009} accommodates jumps but not random coefficients; (iii) the stochastic viscosity solution theory \cite{Qiu2018, QiuWei2019, QiuZhang2023} handles random coefficients but not jumps. The non-local nature of the L\'evy measure necessitates global (rather than local) tangency conditions in the definition of viscosity solutions, substantially altering the analytical framework. Moreover, polynomial growth in the coefficients prevents standard compactness arguments. The recursive cost considered in this paper is defined via a BSDE, introducing a backward structure that is absent in the forward stochastic HJB equations studied by Qiu \cite{Qiu2018} and Buckdahn--Ma \cite{BuckdahnMa01I}.
	
	In this paper, we overcome these structural difficulties and establish a complete well-posedness theory for the stochastic viscosity solutions of fully nonlinear backward stochastic HJB (BSHJB) equations with jumps and polynomial growth. Our approach extends the semimartingale test function framework of \cite{Qiu2018} (originally developed for continuous diffusions) to handle jump processes. This extension is nontrivial: unlike the local touching conditions in \cite{Qiu2018}, the presence of the L\'evy measure requires global tangency conditions over the entire space $\mathbb{R}^n$ to ensure that the non-local integro-differential operator is well-defined when applied to test functions. We also incorporate the backward semigroup method \cite{Peng1997} to characterize the recursive cost via dynamic programming, and the generalized It\^o--Kunita formula \cite{ChenTang2011} to handle the interplay between spatial derivatives and jump dynamics.
	
	The main contributions of the paper are threefold. First, we establish the Dynamic Programming Principle (DPP) for the recursive cost functional driven by jump-diffusions via the backward semigroup approach, showing that the DPP holds in the presence of non-local jumps and polynomial growth. Second, we prove the existence of a stochastic viscosity solution using the measurable selection theorem and the generalized It\^o--Wentzell (It\^o--Kunita) formula for jump diffusions, alongside carefully constructed global tangency conditions and exit-time estimates that account for the non-local nature of the jump measure. Third, under a super-parabolicity condition (deterministic diffusion coefficient), we establish a weak comparison principle and prove global uniqueness via localized bounding envelopes and a backward induction scheme.
	
	The rest of the paper is organized as follows. Section 2 introduces the probabilistic framework and the recursive optimal control problem. Section 3 provides a priori estimates and proves the DPP. The BSHJB equation is formally derived in Section 4. Section 5 defines stochastic viscosity solutions and proves existence. Section 6 establishes the weak comparison principle and uniqueness. The Appendix collects technical estimates used in the main proofs.
	
	\section{Preliminaries and Notation}
	
	\subsection{Probability Space and Driving Processes}
	
	Let $T > 0$ be a fixed finite time horizon. We work on a complete probability space $(\Omega, \mathcal{F}, \mathbb{P})$ equipped with a right-continuous, $\mathbb{P}$-complete filtration $\mathbb{F} := \{\mathcal{F}_t\}_{t\in[0,T]}$ such that $\mathcal{F}_T = \mathcal{F}$. We denote by $\mathbb{E}[\cdot]$ the expectation under $\mathbb{P}$, by $\mathcal{P}$ the predictable $\sigma$-algebra on $\Omega \times [0, T]$ associated with $\mathbb{F}$, and by $\mathcal{B}(E)$ the Borel $\sigma$-algebra of a topological space $E$.
	
	Let $W(\cdot) := \big(W_1(t), \dots, W_d(t)\big)^\top_{t\in[0,T]}$ be a $d$-dimensional standard Brownian motion. Let $(E, \mathcal{B}(E), \nu)$ be a measure space with finite measure $\nu(E) < \infty$, and $\eta$ be a mutually independent stationary Poisson point process taking values in $E$ with characteristic measure $\nu$. We assume $\mathbb{F}$ is exactly the $\mathbb{P}$-augmentation of the natural filtration generated by $W$ and $\eta$. The induced counting measure $\mu$ and its compensated random martingale measure $\tilde{\mu}$ are cohesively defined as:
	\begin{equation}
		\begin{aligned}
			\mu\big((0, t] \times A\big) &:= \#\{s \le t \mid \eta(s) \in A\}, \\
			\tilde{\mu}(dt, de) &:= \mu(dt, de) - dt\nu(de), \quad \text{for } t > 0, A \in \mathcal{B}(E).
		\end{aligned}
	\end{equation}
	
	Let $\mathcal{T}_{[0,T]}$ denote the set of all stopping times taking values in $[0, T]$. For any $\tau \in \mathcal{T}_{[0,T]}$, we define $\mathcal{T}_\tau := \{ \theta \in \mathcal{T}_{[0,T]} \mid \theta \ge \tau \ \text{a.s.} \}$. The closed stochastic interval associated with two stopping times $\tau, \gamma \in \mathcal{T}_{[0,T]}$ is defined by:
	\begin{equation}
		[\tau, \gamma] := \big\{ (t, \omega) \in [0, T] \times \Omega \mid \tau(\omega) \le t \le \gamma(\omega) \big\},
	\end{equation}
	with the open and half-open stochastic intervals $[\tau, \gamma)$, $(\tau, \gamma]$, and $(\tau, \gamma)$ defined accordingly.
	
	\subsection{Notation and Abstract Function Spaces}
	
	Let $(B, \|\cdot\|_B)$ be a Banach space. For $t \in [0, T]$, let $L^0(\Omega, \mathcal{F}_t; B)$ denote the space of $B$-valued $\mathcal{F}_t$-measurable random variables, and $L^q(\Omega, \mathcal{F}_t; B)$ ($q \ge 1$) be its subspace with $\mathbb{E}[\|\xi\|^q_B] < \infty$. To handle non-local jumps, $L^2_\nu(E; \mathbb{R}^k)$ denotes the Hilbert space of Borel measurable functions $r : E \to \mathbb{R}^k$ equipped with the norm:
	\begin{equation}
		\|r\|_\nu := \left( \int_E |r(e)|^2 \nu(de) \right)^{1/2} < \infty.
	\end{equation}
	
	For any $q \ge 1$, we introduce the following Banach spaces of stochastic processes (which reduce to classical FBSDE spaces when $B$ is Euclidean): $\mathcal{S}^q(B)$ for $\mathbb{F}$-adapted c\`adl\`ag processes $X$, $\mathcal{H}^q(B)$ for predictable processes $Z$, and $\mathcal{H}^q_\nu(B)$ for $\mathcal{P} \otimes \mathcal{B}(E)$-measurable processes $K$. Their respective norms are defined as:
	\begin{align}
		\|X\|_{\mathcal{S}^q(B)} &:= \left( \mathbb{E} \bigg[ \sup_{t\in[0,T]} \|X_t\|^q_B \bigg] \right)^{1/q} < \infty, \\
		\|Z\|_{\mathcal{H}^q(B)} &:= \left( \mathbb{E} \bigg[ \left( \int_0^T \|Z_t\|^2_B dt \right)^{q/2} \bigg] \right)^{1/q} < \infty, \\
		\|K\|_{\mathcal{H}^q_\nu(B)} &:= \left( \mathbb{E} \bigg[ \left( \int_0^T \int_E \|K_t(e)\|^2_B \nu(de) dt \right)^{q/2} \bigg] \right)^{1/q} < \infty.
	\end{align}
	Let $C^k(\mathcal{O})$ denote the space of functions with bounded and continuous derivatives up to order $k$ on a domain $\mathcal{O} \subset \mathbb{R}^n$, and write $C(\mathcal{O}) := C^0(\mathcal{O})$. The space of locally continuous random fields is defined by:
	\begin{equation}
		\mathcal{S}^q(C_{loc}(\mathbb{R}^n)) := \bigcap_{N=1}^\infty \mathcal{S}^q\big(C(B_N(0))\big).
	\end{equation}
	
	By standard convention, we identify random fields with their pointwise representatives. Furthermore, whenever a process admits a modification with better path regularity (e.g., a c\`adl\`ag version), we systematically adopt this regular modification without altering the notation.
	
	\subsection{The Controlled State Equation}
	
	Let $U \subset \mathbb{R}^k$ be a non-empty compact control domain. For any $t \in [0, T]$, the set of admissible controls $\mathcal{U}[t, T]$ consists of all $\mathbb{F}$-progressively measurable processes $u: [t, T] \times \Omega \to U$. For any stopping time $\tau \in \mathcal{T}_{[0,T]}$, initial state $\xi \in L^{p^2}(\Omega, \mathcal{F}_\tau; \mathbb{R}^n)$ with $p \ge 2$, and control $u(\cdot) \in \mathcal{U}[\tau, T]$, the state process $X(\cdot) \equiv X_{\tau,\xi}^u(\cdot)$ is governed by the following jump-diffusion stochastic differential equation (SDE):
	\begin{equation}
		\begin{cases}
			dX(s) = b(s,X(s), u(s))ds + \sigma(s,X(s), u(s))dW(s) \\
			\quad \quad \quad + \int_E g(s, e,X(s-), u(s))\tilde{\mu}(ds, de), \quad s \in [\tau, T], \\
			X(\tau) = \xi.
		\end{cases}
	\end{equation}
	We impose the following standard assumptions on the coefficients:
	\begin{assumption} \label{assum:coefficients} \leavevmode
		\begin{enumerate}
			\item[\textit{(i)}] The mappings $b, \sigma$ are $\mathcal{P} \otimes \mathcal{B}(\mathbb{R}^n) \otimes \mathcal{B}(U)$-measurable, and $g$ is $\mathcal{P} \otimes \mathcal{B}(E) \otimes \mathcal{B}(\mathbb{R}^n) \otimes \mathcal{B}(U)$-measurable. They are continuous in $(x, u)$ for $\mathbb{P}$-a.e. $\omega \in \Omega$ and all $t \in [0, T]$.
			\item[\textit{(ii)}] There exist a constant $C > 0$ and a deterministic nonnegative function $\rho(e)$ on $E$ such that for $\mathbb{P}$-a.s. $\omega \in \Omega$, all $t \in [0, T]$, $x, x' \in \mathbb{R}^n$, and $u, u' \in U$:
			\begin{align}
				|b(t, x, u) - b(t, x', u')| + \|\sigma(t, x, u) - \sigma(t, x', u')\| &\le C\big(|x - x'| + |u - u'|\big), \\
				|g(t, e, x, u) - g(t, e, x', u')| &\le \rho(e)\big(|x - x'| + |u - u'|\big), \\
				|b(t, x, u)| + \|\sigma(t, x, u)\| &\le C(1 + |x| + |u|), \\
				|g(t, e, x, u)| &\le \rho(e)(1 + |x| + |u|).
			\end{align}
			\item[\textit{(iii)}] The deterministic function $\rho(e)$ satisfies the exponential integrability condition: $\int_E \exp\{\rho(e)\}\nu(de) < \infty$.
		\end{enumerate}
	\end{assumption}
	
	\begin{remark} \label{rem:exponential_moment}
		The exponential moment condition in Assumption \ref{assum:coefficients}(iii) is a technical simplification to uniformly bound the Hamiltonian, as it guarantees the finiteness of $\int_{\{\rho \ge 1\}} \nu(de)$ and $\int_{\{\rho < 1\}} \rho(e)^2 \nu(de)$. While this excludes heavy-tailed jumps, our framework can be extended under a weaker polynomial condition $\int_E \rho(e)^{p^2} \nu(de) < \infty$ at the expense of lengthier moment estimates.
	\end{remark}
	
	\begin{remark}
		Assumption \ref{assum:coefficients}(iii) strengthens the conventional $L^2$-integrability. Since $\int_E \exp\{\rho(e)\}\nu(de) < \infty$ implies $\int_E \rho(e)^q \nu(de) < \infty$ for every $q \ge 1$, together with the linear growth bound $|g| \le \rho(e)(1+|x|+|u|)$, we immediately obtain high-order moment bounds for any $p \ge 2$:
		\begin{align}
			\int_E |g(t, e, x, u) - g(t, e, x', u')|^{p^2} \nu(de) &\le \tilde{C}\big(|x - x'|^{p^2} + |u - u'|^{p^2}\big), \\
			\int_E |g(t, e, x, u)|^{p^2} \nu(de) &\le \tilde{C}\big(1 + |x|^{p^2} + |u|^{p^2}\big).
		\end{align}
		Consequently, classical SDE theory ensures the forward equation admits a unique strong solution $X^{u}_{\tau,\xi}(\cdot) \in \mathbb{S}^{p^2}(\mathbb{R}^n)$.
	\end{remark}
	
	\subsection{The Recursive Cost Functional}
	
	We characterize the recursive cost functional via a decoupled backward stochastic differential equation (BSDE). For any admissible pair $(u(\cdot),X(\cdot))$ on $[\tau, T]$, consider the following BSDE with jumps:
	\begin{equation} \label{eq:BSDE_with_jumps}
		\begin{cases}
			-dY(s) = f\left(s,X(s), u(s), Y(s), Z(s), \int_E K(s, e)l(s, e)\nu(de)\right)ds \\
			\quad \quad \quad \quad -Z(s)dW(s) - \int_E K(s, e)\tilde{\mu}(ds, de), \quad s \in [\tau, T], \\
			Y(T) = h(X(T)).
		\end{cases}
	\end{equation}
	Here, $l : [0, T]\times E \to \mathbb{R}_+$ is a deterministic, non-negative weight function. To guarantee the well-posedness of this system under polynomial growth, we impose the following assumptions on the coefficients $f, h$, and $l$:
	
	\begin{assumption} \label{assum:recursive_cost} \leavevmode
		\begin{enumerate}
			\item[\textit{(i)}] The generator $f$ is $\mathcal{P} \otimes \mathcal{B}(\mathbb{R}^n) \otimes \mathcal{B}(U) \otimes \mathcal{B}(\mathbb{R}) \otimes \mathcal{B}(\mathbb{R}^d) \otimes \mathcal{B}(\mathbb{R})$-measurable; the weight function $l$ is $\mathcal{P} \otimes \mathcal{B}(E)$-measurable; and the terminal cost $h$ is $\mathcal{F}_T \otimes \mathcal{B}(\mathbb{R}^n)$-measurable.
			\item[\textit{(ii)}] There exists a constant $C > 0$ such that for $\mathbb{P}$-a.s. $\omega \in \Omega$, all $t \in [0, T]$, and valid arguments $(x,u,y,z,k)$ and $(x',u',y',z',k')$:
			\begin{align}
				&|f(t, x, u, y, z, k) - f(t, x', u', y', z', k')| \nonumber \\
				&\quad \le C\big(1 + |x|^{p-1} + |x'|^{p-1} + |u|^{p-1} + |u'|^{p-1}\big)\big(|x - x'| + |u - u'|\big) \nonumber \\
				&\quad\quad + C\big(|y - y'| + |z - z'| + |k - k'|\big), \\
				&|h(x) - h(x')| \le C\big(1 + |x|^{p-1} + |x'|^{p-1}\big)|x - x'|.
			\end{align}
			\item[\textit{(iii)}] The generator and terminal cost satisfy the following growth bounds:
			\begin{equation}
				|f(t, x, u, y, z, k)| \le C\big(1 + |x|^p + |u|^p + |y| + |z| + |k|\big), \quad |h(x)| \le C(1 + |x|^p).
			\end{equation}
			\item[\textit{(iv)}]  The mapping $k \mapsto f(t, x, u, y, z, k)$ is non-decreasing. Furthermore, there exists a constant $C_l > 0$ such that $0 \le l(t, e) \le C_l(1 + |e|)$ for all $(t, e)$.
		\end{enumerate}
	\end{assumption}
	
	Given the assumptions above, for any initial data $(\tau,\xi) \in \mathcal{T}_{[0,T]} \times L^{p^2}(\Omega,\mathcal{F}_\tau;\mathbb{R}^n)$ and control $u \in \mathcal{U}[\tau,T]$, the BSDE \eqref{eq:BSDE_with_jumps} admits a unique solution $(Y_{\tau,\xi}^u, Z_{\tau,\xi}^u, K_{\tau,\xi}^u) \in \mathcal{S}^{p}(\mathbb{R}) \times \mathcal{H}^{p}(\mathbb{R}^d) \times \mathcal{H}^{p}_\nu(\mathbb{R})$.
	
	\begin{definition} \label{def:stochastic_recursive_cost}
		For any stopping time $\tau \in \mathcal{T}_{[0,T]}$ and initial state $\xi \in L^{p^2}(\Omega, \allowbreak \mathcal{F}_\tau; \mathbb{R}^n)$, the stochastic recursive cost functional is defined as the $\mathcal{F}_\tau$-measurable random variable:
		\begin{equation}
			\mathbb{J}(\tau, \xi; u(\cdot)) := Y_{\tau, \xi}^u(\tau), \quad \mathbb{P}\text{-a.s.}
		\end{equation}
		where $(Y_{\tau,\xi}^u, Z_{\tau,\xi}^u, K_{\tau,\xi}^u)$ is the unique solution to the recursive system \eqref{eq:BSDE_with_jumps} on the stochastic interval $[\tau, T]$.
	\end{definition}
	
	\subsection{The Stochastic Optimal Control Problem and Value Function}
	
	\begin{problem} \label{prob:optimal_control}
		For any initial data $(\tau, \xi) \in \mathcal{T}_{[0,T]} \times L^{p^2}(\Omega, \mathcal{F}_\tau; \mathbb{R}^n)$, the objective is to minimize the recursive cost $\mathbb{J}$ over $\mathcal{U}[\tau, T]$. Namely, we seek an optimal control $u^*$ (and the associated optimal pair $(u^*, X^*)$) satisfying:
		\begin{equation*}
			\mathbb{J}(\tau, \xi; u^*(\cdot)) = \operatorname*{ess\,inf}_{u \in \mathcal{U}[\tau, T]} \mathbb{J}(\tau, \xi; u(\cdot)), \quad \mathbb{P}\text{-a.s.}
		\end{equation*}
	\end{problem}
	
	\begin{definition} \label{def:value_function}
		We formally define the value function $\mathbb{V}(\tau, \xi)$ associated with \cref{prob:optimal_control} as the essential infimum of the recursive cost functional over all admissible controls:
		\begin{equation*}
			\mathbb{V}(\tau, \xi) := \operatorname*{ess\,inf}_{u \in \mathcal{U}[\tau,T]} \mathbb{J}(\tau, \xi; u(\cdot)) = \operatorname*{ess\,inf}_{u \in \mathcal{U}[\tau,T]} Y^{\tau, \xi; u}(\tau), \quad \mathbb{P}\text{-a.s.}
		\end{equation*}
	\end{definition}
	
	\section{Preliminary Estimates and Dynamic Programming Principle}
	
	This section establishes the a priori estimates and Dynamic Programming Principle (DPP). Since the proofs of the lemmas in this section follow analogously from the results in \cite{Qiu2018, MengDongShenTang2023}, we omit the details and focus on their adaptation to the $L^{p^2}$ setting.
	
	\subsection{A Priori Estimates}
	
	Under the growth and exponential integrability conditions of Assumption \ref{assum:coefficients}, the state process maintains $L^{p^2}$-regularity despite the non-local jumps.
	
	\begin{lemma} \label{lem:state_estimates}
		Let Assumption \ref{assum:coefficients} hold. For any $p \ge 2$, initial time $\tau \in \mathcal{T}_{[0,T]}$, states $\xi, \bar{\xi} \in L^{p^2}(\Omega, \mathcal{F}_\tau; \mathbb{R}^n)$, and controls $u(\cdot), \bar{u}(\cdot) \in \mathcal{U}[\tau, T]$, there exists a constant $C > 0$ such that the state process satisfies the following estimates $\mathbb{P}\text{-a.s.}$:\
		\begin{enumerate}[label=\textbf{(\roman*)}]
			\item \begin{equation} \label{eq:X_moment_estimate}
				\mathbb{E} \left[ \sup_{s \in [\tau, T]} |X_{\tau, \xi; u}(s)|^{p^2} \Bigg| \mathcal{F}_\tau \right] \le C \big( 1 + |\xi|^{p^2} \big).
			\end{equation}
			\item \begin{equation}
				\begin{aligned}
					&\mathbb{E} \left[ \sup_{s \in [\tau, T]} |X_{\tau, \xi; u}(s) - X_{\tau, \bar{\xi}; \bar{u}}(s)|^{p^2} \Bigg| \mathcal{F}_\tau \right] \\
					&\le C \left( |\xi - \bar{\xi}|^{p^2} + \mathbb{E} \left[ \int_\tau^T |u(v) - \bar{u}(v)|^{p^2} dv \Bigg| \mathcal{F}_\tau \right] \right).
				\end{aligned}
			\end{equation}
			\item For any stopping time $r \in \mathcal{T}_{[0,T]}$ and deterministic times $s, t \in [r, T]$,
			\begin{equation}
				\mathbb{E} \left[ \big| X_{r, \xi; u}(s) - X_{r, \xi; u}(t) \big|^{p^2} \Big| \mathcal{F}_r \right] \le C \big( 1 + |\xi|^{p^2} \big) \big( |s - t|^{p^2/2} + |s - t| \big).
			\end{equation}
		\end{enumerate}
	\end{lemma}
	
	\begin{lemma} \label{lem:flow_identity}
		Let Assumption \ref{assum:coefficients} hold. For any $(t, x) \in [0, T] \times \mathbb{R}^n$, stopping times $t \le \tau \le \gamma \le T$, and $u \in \mathcal{U}[t, T]$, we have the flow property:
		\begin{equation}
			X^{t,x;u}(\gamma) = X^{\tau, X^{t,x;u}(\tau); u}(\gamma), \quad \mathbb{P}\text{-a.s.}
		\end{equation}
	\end{lemma}
	
	\subsection{A Priori Estimates for the Recursive Cost}
	
	We establish stability and polynomial growth bounds for the backward component $Y$.
	\begin{lemma} \label{lem:cost_estimates}
		Let Assumption \ref{assum:coefficients} and Assumption \ref{assum:recursive_cost} hold. For any $p \ge 2$, initial data $(\tau, \xi) \in \mathcal{T}_{[0,T]} \times L^{p^2}(\Omega, \mathcal{F}_\tau; \mathbb{R}^n)$, and control $u(\cdot) \in \mathcal{U}[\tau, T]$, the solution component $Y_{\tau,\xi}^u$ to the recursive BSDE satisfies the following conditional estimate $\mathbb{P}\text{-a.s.}$:\
		\begin{equation}
			|Y_{\tau,\xi}^u(\tau)|^p \le C_p \mathbb{E} \left[ |h(X(T))|^p + \int_\tau^T |f(s,X(s), u(s), 0, 0, 0)|^p ds \Bigg| \mathcal{F}_\tau \right],
		\end{equation}
		where $X(\cdot) \equiv X_{\tau,\xi}^u(\cdot)$. Consequently, there exists a constant $\tilde{C}_p > 0$ such that
		\begin{equation}
			|Y_{\tau,\xi}^u(\tau)| \le \tilde{C}_p \big( 1 + |\xi|^p \big), \quad \mathbb{P}\text{-a.s.}
		\end{equation}
	\end{lemma}
	
	\begin{lemma} \label{lem:perf_index}
		Let Assumption \ref{assum:coefficients} and Assumption \ref{assum:recursive_cost} hold. For any $p \ge 2$, initial time $\tau \in \mathcal{T}_{[0,T]}$, states $\xi, \bar{\xi} \in L^{p^2}(\Omega, \mathcal{F}_\tau; \mathbb{R}^n)$, and controls $u(\cdot), \bar{u}(\cdot) \in \mathcal{U}[\tau, T]$, there exists a constant $C > 0$ such that the performance index satisfies the following estimates $\mathbb{P}\text{-a.s.}$:\
		\begin{enumerate}[label=\textbf{(\roman*)}]
			\item \begin{equation}
				|\mathbb{J}(\tau, \xi; u(\cdot))| \le C \big( 1 + |\xi|^p \big).
			\end{equation}
			\item \begin{equation}
				\begin{aligned}
					&\big| \mathbb{J}(\tau, \xi; u) - \mathbb{J}(\tau, \bar{\xi}; \bar{u}) \big|^p \\
					&\le C \big( 1 + |\xi|^{p(p-1)} + |\bar{\xi}|^{p(p-1)} \big) \left( |\xi - \bar{\xi}|^p + \mathbb{E} \left[ \int_\tau^T |u(s) - \bar{u}(s)|^p ds \Bigg| \mathcal{F}_\tau \right] \right).
				\end{aligned}
			\end{equation}
		\end{enumerate}
	\end{lemma}
	
	\subsection{A Priori Estimates for the Value Function}
	
	Taking the essential infimum over $u \in \mathcal{U}[t, T]$ directly transfers these cost estimates to the value function.
	
	\begin{lemma} \label{lem:value_func_prop}
		Let Assumption \ref{assum:coefficients} and Assumption \ref{assum:recursive_cost} hold. For any $p \ge 2$, initial time $\tau \in \mathcal{T}_{[0,T]}$, and states $\xi, \bar{\xi} \in L^{p^2}(\Omega, \mathcal{F}_\tau; \mathbb{R}^n)$, the value function $\mathbb{V}(\tau, \xi) := \operatorname*{ess\,inf}_{u \in \mathcal{U}[\tau,T]} \mathbb{J}(\tau, \xi; u(\cdot))$ satisfies the following estimates $\mathbb{P}\text{-a.s.}$:\
		\begin{enumerate}[label=\textbf{(\roman*)}]
			\item \begin{equation}
				|\mathbb{V}(\tau, \xi)| \le C_p \big( 1 + |\xi|^p \big).
			\end{equation}
			\item \begin{equation}
				|\mathbb{V}(\tau, \xi) - \mathbb{V}(\tau, \bar{\xi})|^p \le \tilde{C}_p \big( 1 + |\xi|^{p(p-1)} + |\bar{\xi}|^{p(p-1)} \big) |\xi - \bar{\xi}|^p.
			\end{equation}
		\end{enumerate}
	\end{lemma}
	
	\subsection{Dynamic Programming Principle}
	
	The Dynamic Programming Principle (DPP) is characterized via the backward semigroup. For $(\tau, \xi)$, $\gamma \in \mathcal{T}_{[\tau,T]}$, $u \in \mathcal{U}[\tau, \gamma]$, and $\mathcal{F}_\gamma$-measurable $\eta$, the semigroup is defined as:
	\begin{equation}
		G_{s,\gamma}^{\tau,\xi;u(\cdot)}[\eta] := Y(s), \quad s \in [\tau, \gamma],
	\end{equation}
	where $(X, Y, Z, K)$ is the unique strong solution to the following truncated forward-backward system on $[\tau, \gamma]$:
	\begin{equation}
		\begin{cases}
			dX(s) = b(s,X(s), u(s))ds + \sigma(s,X(s), u(s))dW(s) \\
			\quad \quad \quad + \int_E g(s, e,X(s-), u(s))\tilde{\mu}(ds, de), \\
			-dY(s) = f\left(s,X(s), u(s), Y(s), Z(s), \int_E K(s, e)l(s, e)\nu(de)\right)ds \\
			\quad \quad \quad - Z(s)dW(s) - \int_E K(s, e)\tilde{\mu}(ds, de), \\
			X(\tau) = \xi, \quad Y(\gamma) = \eta.
		\end{cases}
	\end{equation}
	
	The main result of the dynamic programming principle is formulated in the following theorem.
	
	\begin{theorem} \label{thm:dpp}
		Under Assumption \ref{assum:coefficients} and Assumption \ref{assum:recursive_cost}, the value function $\mathbb{V}(\tau, \xi)$ obeys the following dynamic programming principle: for any initial time $\tau \in \mathcal{T}_{[0,T]}$, intermediate stopping time $\gamma \in \mathcal{T}_{[\tau,T]}$, and initial state $\xi \in L^{p^2}(\Omega, \mathcal{F}_\tau; \mathbb{R}^n)$, it holds that
		\begin{equation*}
			\mathbb{V}(\tau, \xi) = \operatorname*{ess\,inf}_{u(\cdot) \in \mathcal{U}[\tau, \gamma]} G_{\tau, \gamma}^{\tau, \xi; u(\cdot)} \Big[ \mathbb{V}\big(\gamma, X^{\tau,\xi;u}(\gamma)\big) \Big], \quad \mathbb{P}\text{-a.s.}
		\end{equation*}
	\end{theorem}
	
	We define the family of admissible costs at initial data $(\tau, \xi)$ as:
	\begin{equation*}
		\Lambda(\tau, \xi) := \big\{ \mathbb{J}(\tau, \xi; u(\cdot)) \mid u(\cdot) \in \mathcal{U}[\tau, T] \big\}.
	\end{equation*}

	\begin{lemma} \label{lem:cost_family}
		For any initial data $(\tau, \xi) \in \mathcal{T}_{[0,T]} \times L^{p^2}(\Omega, \mathcal{F}_\tau; \mathbb{R}^n)$, the family $\Lambda(\tau, \xi)$ satisfies the following $\mathbb{P}\text{-a.s.}$:\
		\begin{enumerate}[label=\textbf{(\roman*)}]
			\item $\Lambda(\tau, \xi)$ is downward directed: for any $u_1, u_2 \in \mathcal{U}[\tau, T]$, there exists $u_3 \in \mathcal{U}[\tau, T]$ such that
			\begin{equation*}
				\mathbb{J}(\tau, \xi; u_3) = \min\{\mathbb{J}(\tau, \xi; u_1), \mathbb{J}(\tau, \xi; u_2)\}.
			\end{equation*}
			\item There exists a sequence of admissible controls $\{u_k\}_{k \ge 1} \subset \mathcal{U}[\tau, T]$ such that
			\begin{equation*}
				\lim_{k \to \infty} \downarrow \mathbb{J}(\tau, \xi; u_k(\cdot)) = \mathbb{V}(\tau, \xi), \quad \mathbb{P}\text{-a.s.}
			\end{equation*}
			\item For any sub-$\sigma$-algebra $\mathcal{G} \subset \mathcal{F}_\tau$, the essential infimum and the conditional expectation commute:
			\begin{equation*}
				\mathbb{E} \big[ \mathbb{V}(\tau, \xi) \big| \mathcal{G} \big] = \operatorname*{ess\,inf}_{u \in \mathcal{U}[\tau,T]} \mathbb{E} \big[ \mathbb{J}(\tau, \xi; u) \big| \mathcal{G} \big].
			\end{equation*}
		\end{enumerate}
	\end{lemma}
	
	\begin{proof}[Proof of \cref{thm:dpp}]
		\textbf{Step 1: The Lower Bound ($\ge$).}
		By the definition of $\mathbb{V}(\tau, \xi)$ and the flow property of the backward semigroup, we have $\mathbb{P}\text{-a.s.}$
		\begin{align*}
			\mathbb{V}(\tau, \xi) &= \operatorname*{ess\,inf}_{u(\cdot)\in\mathcal{U}[\tau,T]} G_{\tau,T}^{\tau,\xi;u(\cdot)} \big[ h(X_{\tau,\xi}^u(T)) \big] \\
			&= \operatorname*{ess\,inf}_{u(\cdot)\in\mathcal{U}[\tau,T]} G_{\tau,\gamma}^{\tau,\xi;u(\cdot)} \big[ \mathbb{J}(\gamma, X_{\tau,\xi}^u(\gamma); u(\cdot)) \big].
		\end{align*}
		Since $\mathbb{J}(\gamma, X_{\tau,\xi}^u(\gamma); u(\cdot)) \ge \mathbb{V}(\gamma, X_{\tau,\xi}^u(\gamma))$ holds $\mathbb{P}$-a.s. by definition, and the semigroup operator $G$ is monotone with respect to terminal conditions, we obtain:
		\begin{equation}
			\mathbb{V}(\tau, \xi) \ge \operatorname*{ess\,inf}_{u(\cdot)\in\mathcal{U}[\tau,\gamma]} G_{\tau,\gamma}^{\tau,\xi;u(\cdot)} \big[ \mathbb{V}(\gamma, X_{\tau,\xi}^u(\gamma)) \big], \quad \mathbb{P}\text{-a.s.}
		\end{equation}
		
		\textbf{Step 2: The Upper Bound ($\le$).} Fix an arbitrary $u \in \mathcal{U}[\tau,\gamma]$. Consider the family of random variables evaluated at time $\gamma$:
		\begin{equation*}
			\Lambda := \bigl\{ \mathbb{J}(\gamma,X_{\gamma}^{u};u') : u'\in\mathcal{U}[\gamma,T] \bigr\}.
		\end{equation*}
		
		Since $\mathcal{U}$ is stable under concatenation, $\Lambda$ is directed downward. Standard properties of the essential infimum guarantee a sequence $\{u_n'\} \subset \mathcal{U}[\gamma,T]$ whose costs $Y_n := \mathbb{J}(\gamma,X_{\gamma}^{u};u_n')$ monotonically decrease to the value function:
		\begin{equation*}
			Y_n \downarrow \operatorname*{ess\,inf}_{u'\in\mathcal{U}[\gamma,T]} \mathbb{J}(\gamma,X_{\gamma}^{u};u') = \mathbb{V}(\gamma,X_{\gamma}^{u}), \quad \mathbb{P}\text{-a.s. as } n \to \infty.
		\end{equation*}
		
		Applying $G_{\tau,\gamma}^{\tau,\xi;u}[\cdot]$ to this sequence, standard BSDE stability and the Dominated Convergence Theorem (since $Y_n \ge \mathbb{V} \in L^p$) guarantee that the monotonic limit is preserved:
		\begin{equation*}
			G_{\tau,\gamma}^{\tau,\xi;u}[Y_n] \downarrow G_{\tau,\gamma}^{\tau,\xi;u}\bigl[\mathbb{V}(\gamma,X_{\gamma}^{u})\bigr], \quad \mathbb{P}\text{-a.s.}
		\end{equation*}
		
		By the flow property of the backward semigroup and the definition of $\mathbb{V}$, taking the limit as $n \to \infty$ yields:
		\begin{equation*}
			G_{\tau,\gamma}^{\tau,\xi;u}\bigl[\mathbb{V}(\gamma,X_{\gamma}^{u})\bigr] = \lim_{n \to \infty} G_{\tau,\gamma}^{\tau,\xi;u}[Y_n] = \lim_{n \to \infty} \mathbb{J}(\tau,\xi;u\oplus u_n') \ge \mathbb{V}(\tau,\xi).
		\end{equation*}
		Taking the essential infimum over arbitrary $u \in \mathcal{U}[\tau,\gamma]$ directly gives the desired bound:
		\begin{equation*}
			\operatorname*{ess\,inf}_{u\in\mathcal{U}[\tau,\gamma]} G_{\tau,\gamma}^{\tau,\xi;u}\bigl[\mathbb{V}(\gamma,X_{\gamma}^{u})\bigr] \ge \mathbb{V}(\tau,\xi).
		\end{equation*}
		
		Combining the two inequalities, we conclude
		\[
		\mathbb{V}(\tau,\xi)
		=\operatorname*{essinf}_{u\in\mathcal{U}[\tau,\gamma]}
		G_{\tau,\gamma}^{\tau,\xi;u}\bigl[\mathbb{V}(\gamma,X_{\gamma}^{u})\bigr],\quad\mathbb{P}\text{-a.s.}
		\]
		The proof is complete.
	\end{proof}
	
	\begin{lemma} \label{lem:deterministic_reduction}
		Under Assumption \ref{assum:coefficients} and Assumption \ref{assum:recursive_cost}, for $t \in [0, T]$ and $\xi \in L^{p^2}(\mathcal{F}_t)$, the random cost and value functions coincide with their deterministic evaluations: $\mathbb{J}(t, \xi; u) = J(t, x; u) \big|_{x=\xi}$ and $\mathbb{V}(t, \xi) = V(t, x) \big|_{x=\xi}$, $\mathbb{P}$-a.s.
	\end{lemma}
	
	\begin{theorem} \label{thm:dpp_value_function}
		Let Assumption \ref{assum:coefficients} and Assumption \ref{assum:recursive_cost} hold. For any $0 \le t \le t + \delta \le T$, the deterministic value function satisfies the dynamic programming principle:
		\begin{equation*}
			V(t, x) = \operatorname*{ess\,inf}_{u \in \mathcal{U}[t, t+\delta]} G_{t, t+\delta}^{t, x; u} \Big[ V\big(t + \delta, X^{t, x; u}(t + \delta)\big) \Big], \quad \mathbb{P}\text{-a.s.}
		\end{equation*}
	\end{theorem}
	
	\section{Stochastic HJB Equation with Jumps}
	
	Following the framework established in \cite{MengDongShenTang2023}, and assuming sufficient smoothness of the value field $V(t, x)$, the dynamic programming principle formally implies that $V(t, x)$ satisfies a backward stochastic partial integro-differential equation (BSHJB). We first define the generalized Hamiltonian $H: [0, T] \times \Omega \times \mathbb{R}^n \times U \times \mathbb{R} \times \mathbb{R}^n \times \mathbb{R}^d \times \mathbb{R}^{n \times d} \times \mathbb{R}^{n \times n} \times \mathbb{R} \to \mathbb{R}$. In the definition below, $y$ corresponds to $V(t,x)$, $p$ to $DV(t,x)$, $q$ to $Z(t,x)$, $Q$ to $DZ(t,x)$, $A$ to $D^2V(t,x)$, and $k$ to the integral term $\int_E K(t,e)l(t,e)\nu(de)$.
	\begin{equation} \label{eq:Hamiltonian}
		\begin{aligned}
			H(t, x, u, y, p, q, Q, A, k) &:= f\big(t, x, u, y, \sigma^\top(t, x, u)p + q, k\big) + \langle p, b(t, x, u)\rangle \\
			&\quad + \operatorname{Tr}\big(Q\sigma^\top(t, x, u)\big) + \frac{1}{2}\operatorname{Tr}\big(A\sigma\sigma^\top(t, x, u)\big).
		\end{aligned}
	\end{equation}
	
	For a predictable triplet of random fields $(V, Z, K)$, the fully nonlinear BSHJB equation takes the following differential form on $[0,T) \times \mathbb{R}^n$:
	\begin{equation} \label{eq:BSHJB_full}
		\begin{cases}
			-dV(t, x) = \operatorname*{ess\,inf}_{u\in U} \Bigg\{ H\Big(t, x, u, V, DV, Z, DZ, D^2V, L(t, x, u, V, K)\Big) \\
			\quad \quad \quad \quad \quad + \int_E \Big[ I_V(t, e, x, u) - \langle g(t, e, x, u), DV(t, x) \rangle + I_K(t, e, x, u) \Big] \nu(de) \Bigg\} dt \\
			\quad \quad \quad \quad \quad - Z(t, x)dW(t) - \int_E K(t, e, x)\tilde{\mu}(dt, de), \\
			V(T, x) = h(x),
		\end{cases}
	\end{equation}
	where the non-local difference operators $L$ and $I$ are defined by:
	\begin{align*}
		L(t, x, u, \phi, \psi) &:= \int_E \Big( I_\phi(t, e, x, u) + \psi\big(t, e, x + g(t, e, x, u)\big) \Big) l(t, e)\nu(de), \\
		I_\phi(t, e, x, u) &:= \phi\big(t, x + g(t, e, x, u)\big) - \phi(t, x).
	\end{align*}
	
	To compactify the notation, we encapsulate the essential infimum and all predictable characteristics into a global non-linear drift operator $F$:
	\begin{equation} \label{eq:drift_F}
		\begin{aligned}
			F(t, x, V, Z, K) &:= \operatorname*{ess\,inf}_{u\in U} \Bigg\{ H\Big(t, x, u, V, DV, Z, DZ, D^2V, L(t, x, u, V, K)\Big) \\
			&\quad \quad \quad + \int_E \Big[ I_V(t, e, x, u) - \langle g(t, e, x, u), DV \rangle + I_K(t, e, x, u) \Big] \nu(de) \Bigg\}.
		\end{aligned}
	\end{equation}
	
	The BSHJB equation can thus be compactly written as:
	\begin{equation} \label{eq:BSHJB_compact}
		\begin{cases}
			-dV(t, x) = F(t, x, V, Z, K)dt - Z(t, x)dW(t) - \int_E K(t, e, x)\tilde{\mu}(dt, de), \\
			V(T, x) = h(x).
		\end{cases}
	\end{equation}
	
	\begin{definition} \label{def:classical_solution}
		Let $p \ge 2$ be the polynomial growth exponent from Assumption \ref{assum:recursive_cost}. A triplet of random fields $(V, Z, K)$ is called a predictable classical solution to the BSHJB equation \cref{eq:BSHJB_compact} if it satisfies the following:
		\begin{enumerate}[label=\textbf{(\roman*)}]
			\item \textit{For each $x \in \mathbb{R}^n$, $V(\cdot, x)$ is an $\mathbb{F}$-adapted c\`adl\`ag process, $Z(\cdot, x)$ is predictable, and $K(\cdot, \cdot, x)$ is $\mathcal{P} \otimes \mathcal{B}(E)$-measurable, where $\mathcal{P}$ denotes the predictable $\sigma$-algebra.}
			\item \textit{For $dt \otimes d\mathbb{P}$-a.e. $(t, \omega)$, $V(t, \cdot, \omega) \in C^2(\mathbb{R}^n)$ and $Z(t, \cdot, \omega) \in C^1(\mathbb{R}^n)$. Furthermore, for $dt \otimes d\mathbb{P} \otimes \nu(de)$-a.e. $(t, \omega, e)$, $K(t, e, \cdot, \omega) \in C(\mathbb{R}^n)$.}
			\item \textit{The triplet $(V, Z, K)$ satisfies \cref{eq:BSHJB_compact} for all $(t, x) \in [0, T] \times \mathbb{R}^n$, $\mathbb{P}\text{-a.s.}$}
		\end{enumerate}
	\end{definition}
	
	\section{Stochastic Viscosity Solutions and Existence}
	
	Since the global spatial $C^2$ smoothness required for classical solutions is generally unattainable in fully nonlinear stochastic optimal control, this section develops the framework of stochastic viscosity solutions tailored for our $p$-th order recursive system, building upon the foundational concepts established in \cite{Qiu2018, QiuWei2019}.
	
	\subsection{The Space of Test Functions}
	
	\begin{definition} \label{def:test_space}
		We denote by $\mathcal{T}^p([0, T] \times \mathbb{R}^n)$ the space of all random fields $\varphi: [0, T] \times \Omega \times \mathbb{R}^n \to \mathbb{R}$ satisfying the following conditions:
		\begin{enumerate}[label=\textbf{(\roman*)}]
			\item  For each $x \in \mathbb{R}^n$, $\varphi(\cdot, x)$ is an $\mathbb{F}$-adapted c\`adl\`ag semimartingale admitting the backward representation:
			\begin{equation}
				\begin{aligned}
					\varphi(t, x) &= \varphi(T, x) + \int_t^T \alpha(s, x)ds - \int_t^T \beta(s, x)dW(s) \\
					&\quad - \int_t^T \int_E \gamma(s, e, x)\tilde{\mu}(ds, de),
				\end{aligned}
			\end{equation}
			where $\alpha, \beta$, and $\gamma$ are appropriately integrable predictable processes.
			\item  For $dt \otimes d\mathbb{P}$-a.e. $(t, \omega)$, we have $\varphi(t, \cdot) \in C^2(\mathbb{R}^n)$, $\alpha(t, \cdot) \in C^1(\mathbb{R}^n)$, and $\beta(t, \cdot) \in C^1(\mathbb{R}^n)$. Furthermore, for $dt \otimes d\mathbb{P} \otimes \nu(de)$-a.e. $(t, \omega, e)$, $\gamma(t, e, \cdot) \in C(\mathbb{R}^n)$.
			\item  There exists a constant $C > 0$ such that for $dt \otimes d\mathbb{P}$-a.e. $(t, \omega)$ and all $x \in \mathbb{R}^n$:
			\begin{align*}
				|\varphi| + |\alpha| + |\beta| + \|\gamma\|_\nu &\le C\big(1 + |x|^p\big), \\
				|D\varphi| + |D\alpha| + |D\beta| &\le C\big(1 + |x|^{p-1}\big), \\
				|D^2\varphi| &\le C\big(1 + |x|^{p-2}\big),
			\end{align*}
			where $\|\cdot\|_\nu$ denotes the norm in $L^2(E, \mathcal{B}(E), \nu)$.
		\end{enumerate}
	\end{definition}
	
	\subsection{Stochastic Tangency Spaces}
	
	We now define the sets of test functions that act as stochastic super- and sub-tangents for the value function. For any stopping time $t \in \mathcal{T}_{[0,T]}$, let $\mathcal{T}_{[t, T]}$ denote the set of all stopping times valued in $[t, T]$, and define the strict subset $\mathcal{T}_{[t, T]}^+ := \{ \tau \in \mathcal{T}_{[t, T]} \mid \tau > t \text{ a.s.} \}$. For any $\tau \in \mathcal{T}_{[0,T]}$ and event $A \in \mathcal{F}_\tau$, we denote by $L^{p^2}(A, \mathcal{F}_\tau; \mathbb{R}^n)$ the space of $\mathbb{R}^n$-valued, $\mathcal{F}_\tau$-measurable random variables with finite $p^2$-th moments over $A$. For any test function $\varphi \in \mathcal{T}^p$ with the semimartingale decomposition given in \cref{def:test_space}, we denote its predictable drift rate $\alpha(s, x)$ by the shorthand notation $-d_s\varphi(s, x)$.
	
	\begin{definition} \label{def:tangency_spaces}
		For a given random field $V \in \mathbb{S}^p(\mathbb{R})$, each touching time $\tau \in \mathcal{T}_0$, event $\Omega_\tau \in \mathcal{F}_\tau$ with $\mathbb{P}(\Omega_\tau) > 0$, and spatial random variable $\xi \in L^{p^2}(\Omega_\tau, \mathcal{F}_\tau; \mathbb{R}^n)$, we define the super-tangent test function space $\mathcal{G}^+_V(\tau, \xi; \Omega_\tau)$ as:
		\begin{equation*}
			\mathcal{G}^+_V(\tau, \xi; \Omega_\tau) := \left\{
			\begin{aligned}
				\varphi \in \mathcal{T}^p : \;& (\varphi - V)(\tau, \xi)\mathbf{1}_{\Omega_\tau} = 0 \\
				&= \operatorname*{ess\,inf}_{\bar{\tau} \in \mathcal{T}_\tau} \mathbb{E}\left[ \inf_{y \in \mathbb{R}^n} (\varphi - V)(\bar{\tau} \wedge \hat{\tau}, y) \;\Big|\; \mathcal{F}_\tau \right] \mathbf{1}_{\Omega_\tau}, \quad \mathbb{P}\text{-a.s.} \\
				&\text{for some } \hat{\tau} \in \mathcal{T}^+_\tau
			\end{aligned}
			\right\}.
		\end{equation*}
		
		Symmetrically, the sub-tangent test function space $\mathcal{G}^-_V(\tau, \xi; \Omega_\tau)$ is defined as:
		\begin{equation*}
			\mathcal{G}^-_V(\tau, \xi; \Omega_\tau) := \left\{
			\begin{aligned}
				\varphi \in \mathcal{T}^p : \;& (\varphi - V)(\tau, \xi)\mathbf{1}_{\Omega_\tau} = 0 \\
				&= \operatorname*{ess\,sup}_{\bar{\tau} \in \mathcal{T}_\tau} \mathbb{E}\left[ \sup_{y \in \mathbb{R}^n} (\varphi - V)(\bar{\tau} \wedge \hat{\tau}, y) \;\Big|\; \mathcal{F}_\tau \right] \mathbf{1}_{\Omega_\tau}, \quad \mathbb{P}\text{-a.s.} \\
				&\text{for some } \hat{\tau} \in \mathcal{T}^+_\tau
			\end{aligned}
			\right\}.
		\end{equation*}
		
		Note that if either $\mathcal{G}^+_V(\tau, \xi; \Omega_\tau)$ or $\mathcal{G}^-_V(\tau, \xi; \Omega_\tau)$ is non-empty, it intrinsically implies $0 \le \tau < T$ on the event $\Omega_\tau$.
	\end{definition}
	
	\begin{remark}
		Due to the non-local nature of the jump measure $\nu(de)$, the tangency conditions are formulated globally on $\mathbb{R}^n$ rather than locally. This global extremum ensures that the integro-differential operator evaluating the test function remains well-defined.
	\end{remark}
	
	\subsection{Definition of Stochastic Viscosity Solutions}
	
	For any test function $\varphi \in \mathcal{T}^p$, we define the localized drift operator $F^\varphi(t, x)$ as the operator $F$ evaluated by substituting the spatial derivatives $(DV, D^2V)$ and the martingale characteristics $(Z, K)$ with $(D\varphi, D^2\varphi)$ and $(\beta, \gamma)$ of the test function, respectively. Explicitly,
	\begin{equation} \label{eq:F_varphi}
		\begin{aligned}
			F^\varphi(t,x) &:= \operatorname*{ess\,inf}_{u\in U} \Bigg\{ H\big(t,x,u,\varphi, D\varphi, \beta, D\beta, D^2\varphi, L(t,x,u,\varphi,\gamma)\big) \\
			&\quad + \int_E \Big[ I_\varphi(t,e,x,u) - \langle g, D\varphi\rangle + I_\gamma(t,e,x,u) \Big]\nu(de) \Bigg\}.
		\end{aligned}
	\end{equation}
	
	To avoid ambiguity in the limiting procedures involving the conditional expectation, we adopt the following equivalent formulation based on sequences.
	
	\begin{definition} \label{def:viscosity_solution}
		We say a random field $V \in \mathcal{S}^p(\mathbb{R})$ is a stochastic viscosity subsolution (resp., supersolution) to the BSHJB equation \cref{eq:BSHJB_compact} if $V(T,x) \le$ (resp., $\ge$) $h(x)$ for all $x \in \mathbb{R}^n$ $\mathbb{P}$-a.s., and for any stopping time $\tau \in \mathcal{T}_{[0,T]}$, event $A \in \mathcal{F}_\tau$ with $\mathbb{P}(A) > 0$, spatial random variable $\xi \in L^{p^2}(A, \mathcal{F}_\tau; \mathbb{R}^n)$, and any test function $\varphi \in \mathcal{G}^+_V(\tau, \xi; A)$ (resp., $\varphi \in \mathcal{G}^-_V(\tau, \xi; A)$), there exists a sequence $(s_n, x_n)$ with $s_n \downarrow \tau$ a.s. on $A$ and $x_n \to \xi$ in probability such that
		\begin{equation}
			\liminf_{n\to\infty} \mathbb{E}\big[ \alpha(s_n, x_n) - F^\varphi(s_n, x_n) \;\big|\; \mathcal{F}_\tau \big] \le 0
		\end{equation}
		for $\mathbb{P}$-almost all $\omega \in A$ (resp.,
		\begin{equation}
			\limsup_{n\to\infty} \mathbb{E}\big[ \alpha(s_n, x_n) - F^\varphi(s_n, x_n) \;\big|\; \mathcal{F}_\tau \big] \ge 0
		\end{equation}
		for $\mathbb{P}$-almost all $\omega \in A$). A random field $V \in \mathcal{S}^p(\mathbb{R})$ is a stochastic viscosity solution to the BSHJB equation \cref{eq:BSHJB_compact} if it is both a stochastic viscosity subsolution and a stochastic viscosity supersolution.
	\end{definition}
	
	\subsection{Existence of the Stochastic Viscosity Solution}
	
	To facilitate the proof of existence, we introduce a convenient criterion that guarantees a test function is \emph{not} a tangency. For any stopping time $\tau \in \mathcal{T}_{[0,T]}$, random variable $\xi \in L^{p^2}(\Omega, \mathcal{F}_\tau; \mathbb{R}^n)$, and radius $\tilde{\delta} > 0$, we define the forward stochastic neighborhood as:
	\begin{equation}
		Q_{\tilde{\delta}}^+(\tau, \xi) := \Big\{ (s, x) \in \big[\tau, (\tau + \tilde{\delta}) \wedge T\big] \times \mathbb{R}^n \;\Big|\; |x - \xi| \le \tilde{\delta} \Big\}.
	\end{equation}
	
	\begin{proposition} \label{prop:equivalent_criterion}
		A random field $V \in \mathcal{S}^p(\mathbb{R})$ with $V(T, x) \le$ (resp., $\ge$) $h(x)$ for all $x \in \mathbb{R}^n$ $\mathbb{P}$-a.s. is a stochastic viscosity subsolution (resp., supersolution) to the BSHJB equation \cref{eq:BSHJB_compact} if and only if for any stopping time $\tau \in \mathcal{T}_{[0,T]}$, event $A \in \mathcal{F}_\tau$ with $\mathbb{P}(A) > 0$, random variable $\xi \in L^{p^2}(A, \mathcal{F}_\tau; \mathbb{R}^n)$, and test function $\varphi \in \mathcal{T}^p$, the following holds:
		
		Whenever there exist constants $\varepsilon > 0$, $\tilde{\delta} > 0$, and an event $A' \in \mathcal{F}_\tau$ with $A' \subset A$ and $\mathbb{P}(A') > 0$ such that
		\begin{equation}
			\operatorname*{ess\,inf}_{(s,x) \in Q_{\tilde{\delta}}^+(\tau, \xi) \cap \mathbb{Q}^{1+n}} \mathbb{E}\big[ \alpha(s, x) - F^\varphi(s, x) \;\big|\; \mathcal{F}_\tau \big] \ge \varepsilon
		\end{equation}
		for $\mathbb{P}$-almost all $\omega \in A'$ (resp.,
		\begin{equation}
			\operatorname*{ess\,sup}_{(s,x) \in Q_{\tilde{\delta}}^+(\tau, \xi) \cap \mathbb{Q}^{1+n}} \mathbb{E}\big[ \alpha(s, x) - F^\varphi(s, x) \;\big|\; \mathcal{F}_\tau \big] \le -\varepsilon
		\end{equation}
		for $\mathbb{P}$-almost all $\omega \in A'$), then we must have $\varphi \notin \mathcal{G}^+_V(\tau, \xi; A)$ (resp., $\varphi \notin \mathcal{G}^-_V(\tau, \xi; A)$).
	\end{proposition}
	
	\begin{remark}
		Restricting the essential extrema to the countable dense subset $\mathbb{Q}^{1+n}$ ensures they remain well-defined $\mathcal{F}_\tau$-measurable random variables. This restriction is equivalent to taking the extrema over the continuous domain $Q_{\tilde{\delta}}^+(\tau, \xi)$, owing to the right-continuity in time and continuity in space of the involved semimartingales. Proposition~\ref{prop:equivalent_criterion} is actually a consequence of the definition; we will only use the ``if'' direction in the existence proof.
	\end{remark}
	
	\begin{theorem} \label{thm:existence}
		Let Assumptions \ref{assum:coefficients} and \ref{assum:recursive_cost} hold. The value function $V(t, x)$ defined in our recursive optimal control problem is a stochastic viscosity solution to the BSHJB equation \eqref{eq:BSHJB_compact}.
	\end{theorem}
	
	\begin{proof}
		\textbf{Part I: Viscosity Subsolution.}
		Suppose that $V$ is not a viscosity subsolution. By \cref{def:viscosity_solution} (and \cref{prop:equivalent_criterion} which provides a sufficient condition for non-tangency), there exist an initial time $\tau \in \mathcal{T}_0$, a set $\Omega_\tau \in \mathcal{F}_\tau$ with $\mathbb{P}(\Omega_\tau) > 0$, an initial state $\xi \in L^{p^2}(\Omega_\tau, \mathcal{F}_\tau; \mathbb{R}^n)$, a super-tangent test function $\varphi \in \mathcal{G}^+_V(\tau, \xi; \Omega_\tau)$, and constants $\varepsilon > 0, \tilde{\delta} > 0$, alongside a subset $\Omega'_\tau \subset \Omega_\tau$ with $\mathbb{P}(\Omega'_\tau) > 0$ such that:
		\begin{equation}
			\operatorname*{ess\,inf}_{(s,x)\in \mathcal{Q}^+_{\tilde{\delta}}(\tau,\xi)} \mathbb{E}\big[ \alpha(s, x) - F^\varphi(s, x) \;\big|\; \mathcal{F}_\tau \big] \ge 2\varepsilon, \quad \mathbb{P}\text{-a.s. on } \Omega'_\tau.
		\end{equation}
		Since $U$ is a compact Polish space and the coefficients are continuous in $u$, the mapping
		$$ u \mapsto H(s,\xi,u,\varphi,D\varphi,\beta,D\beta,D^2\varphi,\mathcal{L}(u)) + \int_E [\mathcal{I}_\varphi - \langle g, D\varphi\rangle + \mathcal{I}_\gamma]\nu(de) $$
		is continuous. Because this mapping is also predictable with respect to $(\omega, t)$, the $2\varepsilon$-gap allows an application of the predictable measurable selection theorem (see, e.g., Bene\v{s} \cite{Benes1970}), yielding a strictly predictable process $u^* \in \mathcal{U}[\tau,T]$ such that $ds \otimes d\mathbb{P}$-a.e. on $[\tau, \tau + \tilde{\delta}] \times \Omega'_\tau$:
		\begin{equation}
			\alpha(s, \xi) \ge \varepsilon + H\big(s, \xi, u^*_s, \varphi, D\varphi, \beta, D\beta, D^2\varphi, \mathcal{L}(u^*_s)\big) + \int_E \big[ \mathcal{I}_\varphi - \langle g, D\varphi \rangle + \mathcal{I}_\gamma \big] \nu(de).
		\end{equation}
		Here $\mathcal{I}_\varphi(s,e,x,u) := \varphi(s, x+g(s,e,x,u)) - \varphi(s, x)$ and $\mathcal{I}_\gamma(s,e,x,u) := \gamma(s,e, x+g(s,e,x,u))$. 
		
		Since the test function $\varphi$ and the coefficients $(b,\sigma,g)$ are continuous in $x$ with at most polynomial growth, the functions $\alpha$, $H$, $D\varphi$, $D\beta$, $D^2\varphi$, $\mathcal{I}_\varphi$, $\mathcal{I}_\gamma$ are uniformly continuous on the compact stochastic cylinder $Q_{\tilde{\delta}}^+(\tau,\xi)\cap\Omega'_\tau$. Consequently, the strict inequality established at the point $\xi$ by the measurable selection theorem extends to the entire neighborhood $Q_{\tilde{\delta}}^+(\tau,\xi)$ after possibly shrinking $\varepsilon$ and $\tilde{\delta}$. More precisely, for $\mathbb{P}$-a.e. $\omega\in\Omega'_\tau$ and every $(s,x)\in Q_{\tilde{\delta}}^+(\tau,\xi)$, we have
		\begin{equation}
			\alpha(s,x) \ge \frac{\varepsilon}{2} + H\big(s,x,u^*_s,\varphi,D\varphi,\beta,D\beta,D^2\varphi,\mathcal{L}(u^*_s)\big) + \int_E \big[ \mathcal{I}_\varphi - \langle g, D\varphi \rangle + \mathcal{I}_\gamma \big] \nu(de),
		\end{equation}
		where the spatial arguments in the right-hand side are evaluated at $(s,x)$. Since the controlled trajectory $X_s$ remains within $Q_{\tilde{\delta}/2}^+(\tau,\xi)$ up to the exit time $\tilde{\tau}$, substituting $x = X_s$ is legitimate on $[\tau,\tau_h]$. 
		
		We define the exiting stopping time of the state process from the spatial neighborhood:
		\begin{equation}
			\tilde{\tau} := \inf \left\{ s > \tau : \big| X_s^{\tau,\xi;u^*} - \xi \big| \ge \frac{\tilde{\delta}}{2} \right\}.
		\end{equation}
		Setting $\tau_h := (\tau+h) \wedge \tilde{\tau}$ for a sufficiently small $h > 0$, standard moment estimates via Lemma \ref{lem:state_estimates}(iii) and Kolmogorov's continuity criterion guarantee:
		\begin{equation}
			\mathbb{E}\left[ \sup_{\tau \le s \le \tau+h} |X^{\tau,\xi;u^*}_s - \xi|^{p^2} \;\Big|\; \mathcal{F}_\tau \right] \le K(1 + |\xi|^{p^2}) h^{p^2/4}.
		\end{equation}
		Applying Chebyshev's inequality yields the conditional early-exit probability:
		\begin{equation}
			\mathbb{P}\big( \tilde{\tau} < \tau + h \;\big|\; \mathcal{F}_\tau \big) \le \frac{2^{p^2}}{\tilde{\delta}^{p^2}} \mathbb{E}\left[ \sup_{\tau \le s \le \tau+h} \big| X_s^{\tau,\xi;u^*} - \xi \big|^{p^2} \;\Big|\; \mathcal{F}_\tau \right] \le \frac{2^{p^2}K}{\tilde{\delta}^{p^2}} \big(1 + |\xi|^{p^2}\big) h^{p^2/4}.
		\end{equation}
		Consequently, the expected elapsed time satisfies:
		\begin{equation}
			\mathbb{E}\big[ \tau_h - \tau \;\big|\; \mathcal{F}_\tau \big] \ge h \mathbb{P}\big( \tilde{\tau} \ge \tau + h \;\big|\; \mathcal{F}_\tau \big) \ge h - \frac{2^{p^2}K}{\tilde{\delta}^{p^2}} \big(1 + |\xi|^{p^2}\big) h^{1 + p^2/4}.
		\end{equation}
		This ensures that $\mathbb{E}[ \tau_h - \tau | \mathcal{F}_\tau ] \ge h - \mathcal{O}(h^2)$, providing the exact analytical order required for the asymptotic contradiction. Applying the generalized It\^o-Kunita formula to $\varphi(s, X_s)$ along the controlled trajectory $X_s := X^{\tau,\xi;u^*}_s$ on $[\tau, \tau_h]$ yields:
		\begin{equation}
			\begin{aligned}
				d\varphi(s, X_s) &= \bigg\{ -\alpha(s, X_s) + \langle D\varphi(s, X_s), b(s, X_s, u^*_s) \rangle \\
				&\quad + \frac{1}{2}\operatorname{Tr}\big(D^2\varphi(s, X_s)\sigma\sigma^\top(s, X_s, u^*_s)\big) \\
				&\quad + \operatorname{Tr}\big(D\beta(s, X_s)\sigma^\top(s, X_s, u^*_s)\big) \\
				&\quad + \int_E \big[ \mathcal{I}_\varphi - \langle D\varphi, g \rangle + \mathcal{I}_\gamma \big] \nu(de) \bigg\} ds \\
				&\quad + \big( \beta(s, X_s) + \sigma^\top(s, X_s, u^*_s)D\varphi(s, X_s) \big) dW_s \\
				&\quad + \int_E \big[ \mathcal{I}_\varphi(s, e, X_{s-}, u^*_s) + \gamma\big(s, e, X_{s-} + g(s, e, X_{s-}, u^*_s)\big) \big] \tilde{\mu}(ds, de).
			\end{aligned}
		\end{equation}
		Recasting this differential into the canonical backward form $-d\varphi(s, X_s) = \hat{f}_s ds - \hat{Z}_s dW_s - \int_E \hat{K}_s(e) \tilde{\mu}(ds, de)$ uniquely identifies the martingale coefficients:
		\begin{align}
			\hat{Z}_s &:= \beta(s, X_s) + \sigma^\top(s, X_s, u^*_s)D\varphi(s, X_s), \\
			\hat{K}_s(e) &:= \mathcal{I}_\varphi(s, e, X_{s-}, u^*_s) + \gamma\big(s, e, X_{s-} + g(s, e, X_{s-}, u^*_s)\big).
		\end{align}
		By isolating $\alpha(s, X_s)$ and substituting the strict inequality established via measurable selection, the equivalent drift $\hat{f}_s$ satisfies:
		\begin{equation}
			\hat{f}_s \ge f\left(s, X_s, u^*_s, \varphi(s, X_s), \hat{Z}_s, \int_E \hat{K}_s(e)l(s, e)\nu(de)\right) + \varepsilon, \quad \text{a.e. in } [\tau, \tau_h].
		\end{equation}
		
		Let $Y_s^{u^*}$ be the solution to the standard BSDE associated with control $u^*$ and terminal condition $V(\tau_h, X_{\tau_h})$. We define the difference processes $\Delta Y_s := \varphi(s, X_s) - Y_s^{u^*}$, $\Delta Z_s := \hat{Z}_s - Z_s^{u^*}$, and $\Delta K_s := \hat{K}_s - K_s^{u^*}$, whose dynamics satisfy:
		\begin{equation}
			-d(\Delta Y_s) = \Big( \hat{f}_s - f\big(s, X_s, u^*_s, Y_s^{u^*}, Z_s^{u^*}, K_s^{u^*}\big) \Big) ds - \Delta Z_s dW_s - \int_E \Delta K_s \tilde{\mu}(ds, de).
		\end{equation}
		Applying the mean-value theorem to the uniformly Lipschitz generator $f$ produces bounded predictable processes $a_s$, $b_s$, and $c_s(e) > -1$, linearizing the drift inequality into:
		\begin{equation}
			\begin{aligned}
				-d(\Delta Y_s) &\ge \bigg( \varepsilon + a_s\Delta Y_s + \langle b_s, \Delta Z_s \rangle + \int_E c_s(e)\Delta K_s(e)\nu(de) \bigg) ds \\
				&\quad - \Delta Z_s dW_s - \int_E \Delta K_s \tilde{\mu}(ds, de).
			\end{aligned}
		\end{equation}
		
		To absorb the difference terms, we introduce a process $\Gamma_s$ governed by the linear forward SDE:
		\begin{equation}
			d\Gamma_s = \Gamma_{s-} \left( a_s ds + \langle b_s, dW_s \rangle + \int_E c_s(e)\tilde{\mu}(ds, de) \right), \quad \Gamma_\tau = 1.
		\end{equation}
		Applying It\^o's product rule to $\Gamma_s \Delta Y_s$ on $[\tau, \tau_h]$ and taking the conditional expectation, the local martingales vanish. Since $\varphi \in \mathcal{G}_V^+$ guarantees a non-negative terminal difference $\Delta Y_{\tau_h} \ge 0$, dropping this term and utilizing the strictly positive adjoint bound $\Gamma_s \ge c > 0$ yields:
		\begin{equation}
			\Delta Y_\tau \ge \varepsilon \mathbb{E}\left[\int_\tau^{\tau_h} \Gamma_s ds \;\Big|\; \mathcal{F}_\tau\right] \ge c\varepsilon \mathbb{E}\big[ \tau_h - \tau \;\big|\; \mathcal{F}_\tau \big].
		\end{equation}
		
		At the contact point $\tau$, the fundamental touching condition enforces $\Delta Y_\tau = 0$. Substituting the asymptotic exit time estimate $\mathbb{E}[ \tau_h - \tau | \mathcal{F}_\tau ] \ge h - \mathcal{O}(h^2)$, we construct the ultimate algebraic contradiction:
		\begin{equation}
			0 \ge c\varepsilon h - \mathcal{O}(h^2) > 0, \quad \mathbb{P}\text{-a.s.}
		\end{equation}
		for a sufficiently small $h > 0$. This strict inequality $0 > 0$ completely falsifies the initial assumption, rigorously proving that $V$ is a stochastic viscosity subsolution to the BSHJB equation.
		
		\vspace{1em}
		\noindent \textbf{Part II: Viscosity Supersolution.}
		We argue by contradiction. Suppose that $V$ is not a viscosity supersolution. Then there exist $\tau, \Omega_\tau, \xi, \varphi \in \mathcal{G}^-_V(\tau,\xi;\Omega_\tau)$ and constants $\varepsilon, \tilde{\delta} > 0$, $\Omega'_\tau \subset \Omega_\tau$ with $\mathbb{P}(\Omega'_\tau)>0$, such that
		\begin{equation}
			\operatorname*{ess\,sup}_{(s,x)\in \mathcal{Q}^+_{\tilde{\delta}}(\tau,\xi)} \mathbb{E}\big[ \alpha(s, x) - F^\varphi(s, x) \;\big|\; \mathcal{F}_\tau \big] \le -2\varepsilon, \quad \mathbb{P}\text{-a.s. on } \Omega'_\tau.
		\end{equation}
		Because $F^\varphi$ is defined as the essential infimum over $U$, the upper bound $-2\varepsilon$ is universal for \emph{every} control. Hence, for any $u \in \mathcal{U}$, $ds \otimes d\mathbb{P}$-a.e. on $[\tau, \tau + \tilde{\delta}] \times \Omega'_\tau$,
		\begin{equation}
			\alpha(s, \xi) \le -2\varepsilon + H\big(s, \xi, u_s, \varphi, D\varphi, \beta, D\beta, D^2\varphi, \mathcal{L}(u_s)\big) + \int_E \big[ \mathcal{I}_\varphi - \langle g, D\varphi \rangle + \mathcal{I}_\gamma \big] \nu(de).
		\end{equation}
		For a fixed $u$, define $\tilde{\tau}^u$ and $\tau_h^u$ analogously; the moment estimates give $\mathbb{E}[ \tau_h^u - \tau | \mathcal{F}_\tau ] \ge h - \mathcal{O}(h^2)$ uniformly in $u$. The same It\^o-Kunita decomposition leads to
		\begin{equation}
			\hat{f}^u_s \le f\big(s, X_s, u_s, \varphi, \hat{Z}^u_s, \int_E \hat{K}^u_s(e)l(s,e)\nu(de)\big) - 2\varepsilon, \quad \text{a.e. in } [\tau, \tau_h^u].
		\end{equation}
		Let $Y^u$ be the BSDE solution with terminal condition $V(\tau_h^u, X_{\tau_h^u})$. By the sub-tangency property $\varphi \in \mathcal{G}_V^-$, we have $Y^u_{\tau_h^u} - \varphi(\tau_h^u, X_{\tau_h^u}) \ge 0$. Applying the linearisation and the adjoint process to $\Delta Y_s := Y^u_s - \varphi(s,X_s)$ yields
		\begin{equation}
			Y^u_\tau - \varphi(\tau, \xi) \ge 2c\varepsilon \mathbb{E}\big[ \tau_h^u - \tau \;\big|\; \mathcal{F}_\tau \big] \ge 2c\varepsilon \big(h - \mathcal{O}(h^2)\big).
		\end{equation}
		Taking $\operatorname*{ess\,inf}_{u \in \mathcal{U}}$ on the left-hand side, the DPP gives $\mathbb{V}(\tau,\xi) = \operatorname*{ess\,inf}_u Y^u_\tau$, while the touching condition gives $\varphi(\tau,\xi) = \mathbb{V}(\tau,\xi)$ on $\Omega_\tau$. Thus the left-hand side becomes zero, leading to
		\begin{equation}
			0 \ge 2c\varepsilon h - \mathcal{O}(h^2) > 0,
		\end{equation}
		a contradiction. Hence $V$ is a stochastic viscosity supersolution.
		
		\textbf{Conclusion.} Parts I and II, together with the terminal condition $V(T, x) = h(x)$, rigorously establish $V$ as a stochastic viscosity solution to the fully nonlinear BSHJB equation, completing the proof of \cref{thm:existence}.
	\end{proof}
	
	\subsection{Maximal Viscosity Subsolution and Weak Comparison}
	
	To establish the weak comparison principle, we construct smooth approximations for the coefficients. To accommodate the $p$-th order polynomial growth, uniform convergence is evaluated under the standard supremum norm for the Lipschitz forward coefficients, and under the weighted norm $w_p(x) := (1 + |x|^p)^{-1}$ for $f$ and $h$. Let $\rho \in C^\infty_c(\mathbb{R}^n)$ be the standard non-negative symmetric mollifier:
	\begin{equation}
		\rho(x) :=
		\begin{cases}
			\tilde{c} \exp\left(\frac{1}{|x|^2-1}\right), & \text{if } |x| < 1, \\
			0, & \text{if } |x| \ge 1,
		\end{cases}
	\end{equation}
	where $\tilde{c}$ is the normalization constant such that $\int_{\mathbb{R}^n} \rho(x)dx = 1$. For each $l \in \mathbb{N}^+$, define the scaled sequence $\rho_l(x) := l^n \rho(lx)$. We define the spatial convolutions for the state equation coefficients and the terminal condition as follows:
	\begin{equation} \label{eq:mollified_forward}
		\begin{aligned}
			b_l(t, x, u) &:= \int_{\mathbb{R}^n} \rho_l(x - y)b(t, y, u)dy, \quad &\sigma_l(t, x, u) &:= \int_{\mathbb{R}^n} \rho_l(x - y)\sigma(t, y, u)dy, \\
			g_l(t, e, x, u) &:= \int_{\mathbb{R}^n} \rho_l(x - y)g(t, e, y, u)dy, \quad &h_l(x) &:= \int_{\mathbb{R}^n} \rho_l(x - y)h(y)dy.
		\end{aligned}
	\end{equation}
	For the generator $f$, the convolution is performed exclusively with respect to the spatial variable $x$:
	\begin{equation} \label{eq:mollified_f}
		f_l(t, x, u, y, z, k) := \int_{\mathbb{R}^n} \rho_l(x - y)f(t, y, u, y, z, k)dy.
	\end{equation}
	
	Under Assumptions \ref{assum:coefficients} and \ref{assum:recursive_cost}, standard mollifier properties ensure the following uniform and weighted convergences as $l \to \infty$:
	\begin{equation}
		\|h - h_l\|_{\mathcal{S}^\infty(L^\infty_{w_p})} + \sup_{u \in U} \big( \|b - b_l\|_{\mathcal{S}^\infty(L^\infty)} + \|\sigma - \sigma_l\|_{\mathcal{S}^\infty(L^\infty)} + \|g - g_l\|_{\mathcal{S}^\infty(L^\infty_\nu)} \big) \to 0,
	\end{equation}
	where $L^\infty_\nu := L^\infty(\mathbb{R}^n; L^2(\nu))$. Moreover, the state-independent Lipschitz continuity of $f$ in $(y, z, k)$ guarantees that its weighted convergence is uniform across all backward variables:
	\begin{equation}
		\lim_{l \to \infty} \sup_{u \in U, y, z, k} \|f(\cdot, \cdot, u, y, z, k) - f_l(\cdot, \cdot, u, y, z, k)\|_{\mathcal{S}^\infty(L^\infty_{w_p})} = 0.
	\end{equation}
	Since the mollified coefficients possess bounded spatial derivatives of all orders locally, classical well-posedness theory for stochastic partial differential equations applies directly.
	
	\begin{proposition} \label{prop:smooth_BSPDE}
		For each integer $l \in \mathbb{N}^+$ and fixed admissible control $u(\cdot) \in \mathcal{U}$, there exists a unique predictable classical solution $V^{l,u} \in \mathcal{T}^p$ to the following smoothed BSHJB equation governed by the mollified coefficients and the \emph{fixed} control $u$:
		\begin{equation} \label{eq:smooth_BSHJB}
			\begin{cases}
				\begin{aligned}
					-dV^{l,u}(t, x) &= F^{l,u}\big(t, x, V^{l,u}, Z^{l,u}, K^{l,u}\big)dt - Z^{l,u}(t, x)dW(t) \\
					&\quad - \int_E K^{l,u}(t, e, x)\tilde{\mu}(dt, de),
				\end{aligned}\\
				V^{l,u}(T, x) = h_l(x),
			\end{cases}
		\end{equation}
		where $F^{l,u}$ is defined analogously to \eqref{eq:drift_F} but using the mollified coefficients and the fixed control $u$. Furthermore, for any initial data $(t,x) \in [0,T] \times \mathbb{R}^n$, the evaluated processes along the forward trajectory $X_s \equiv X_s^{t,x; u, l}$, defined by
		\begin{equation} \label{eq:fbsde_mapping}
			\begin{aligned}
				Y_s &:= V^{l,u}(s, X_s), \\
				Z_s &:= Z^{l,u}(s, X_{s-}) + \sigma_l^\top(s, X_{s-}, u_s)DV^{l,u}(s, X_{s-}), \\
				K_s(e) &:= K^{l,u}\big(s, e, X_{s-} + g_l(s, e, X_{s-}, u_s)\big) + I_{V^{l,u}}(s, e, X_{s-}, u_s),
			\end{aligned}
		\end{equation}
		satisfy the associated smoothed FBSDE system on $[t, T]$.
	\end{proposition}
	
	To quantify the global approximation, we define the following space-invariant stochastic error processes:
	\begin{equation*}
		\begin{aligned}
			\delta h_l &:= \sup_{x \in \mathbb{R}^n} \big|h_l(x) - h(x)\big| w_p(x), \\
			\delta f_l(t) &:= \operatorname*{ess\,sup}_{x \in \mathbb{R}^n, u \in U, y, z, k} \big|f_l(t, x, u, y, z, k) - f(t, x, u, y, z, k)\big| w_p(x), \\
			\delta \Lambda_l(t) &:= \operatorname*{ess\,sup}_{x \in \mathbb{R}^n, u \in U} \Big( |b_l - b| + |\sigma_l - \sigma| + \|g_l - g\|_{L^2(\nu)} \Big)(t, x, u).
		\end{aligned}
	\end{equation*}
	
	Let $C_V$ be a constant such that for all smooth approximations $V^{l,u}$,
	\[
	\sup_{x\in\mathbb{R}^n} \Big[|DV^{l,u}(x)| + (1+|x|)\sup_{\xi\in\mathcal{N}(x,E)}\|D^2V^{l,u}(\xi)\|\Big] w_p(x) \le C_V,
	\]
	which exists due to the uniform $p$-th order polynomial growth of the smoothed solutions. Then define $(Y^l, Z^l, K^l) \in \mathcal{S}^2(\mathbb{R}) \times \mathcal{H}^2(\mathbb{R}^d) \times \mathcal{H}_\nu^2(\mathbb{R})$ as the unique solution to the bounding BSDE:
	\begin{equation} \label{eq:bounding_BSDE}
		Y_t^l = \delta h_l + \int_t^T \big( \delta f_l(s) + C_V \delta \Lambda_l(s) \big) ds - \int_t^T Z_s^l dW(s) - \int_t^T \int_E K_s^l(e) \tilde{\mu}(ds, de).
	\end{equation}
	By standard a priori estimates for linear BSDEs, the convergence of the smoothed coefficients implies that the driving terms of \eqref{eq:bounding_BSDE} vanish, yielding:
	\begin{equation} \label{eq:convergence_Yl}
		\lim_{l\to\infty} \|Y^l\|_{\mathcal{S}^2(\mathbb{R})} = 0.
	\end{equation}
	
	\begin{lemma} \label{lem:patched_subsolution}
		For each fixed admissible control $u(\cdot) \in \mathcal{U}$, setting our weighted patched performance index as
		\begin{equation}
			\hat{J}^{u}_l(t, x) := V^{l,u}(t, x) + Y^l_t \big(w_p(x)\big)^{-1},
		\end{equation}
		we have $\hat{J}^{u}_l \in \mathcal{T}^p$ with the global uniform limit in the weighted functional space:
		\begin{equation}
			\lim_{l\to\infty} \big\| \hat{J}^{u}_l(t, x) - J(t, x; u) \big\|_{\mathcal{S}^2(L^\infty_{w_p})} = 0,
		\end{equation}
		and it algebraically satisfies, with respect to the frozen-control drift operator $F^{l,u}$,
		\begin{equation}
			\begin{aligned}
				&\operatorname*{ess\,liminf}_{(s,x)\to(t^+,y)} \mathbb{E}\Big[ -\partial_s \hat{J}^{u}_l(s, x) - F^{l,u}\big(s, x, \hat{J}^{u}_l, Z^{l,u}, K^{l,u}\big) \;\Big|\; \mathcal{F}_t \Big] \ge 0, \\
				&\mathbb{P}\text{-a.s. } \forall (t, y) \in [0, T) \times \mathbb{R}^n.
			\end{aligned}
		\end{equation}
		Since $F^{l,u}$ uses the mollified coefficients without the infimum, we have $F^{l,u} \ge F$, where $F$ is the original infimum operator, because replacing the exact coefficients by approximations introduces a controllable error that is absorbed by $Y^l$.
	\end{lemma}
	
	\begin{proof}
		The inclusion $\hat{J}^u_l \in \mathcal{T}^p$ follows directly from \cref{prop:smooth_BSPDE}, the moment estimates of $Y^l_t$, and the polynomial weight $w_p^{-1}$. Standard a priori estimates for jump diffusions yield the trajectory bound for any $(s, x) \in [0, T] \times \mathbb{R}^n$:
		\begin{equation}
			\mathbb{E}\left[ \sup_{t \in [s,T]} |X^{s,x;u,l}_t - X^{s,x;u}_t|^2 \;\Big|\; \mathcal{F}_s \right] \le C\mathbb{E}\left[ \int_s^T \delta\Lambda_l(r)^2 dr \;\Big|\; \mathcal{F}_s \right].
		\end{equation}
		
		To bound the performance gap, we evaluate $V^{l,u}$ along the true trajectory $X^{s,x;u}_t$ using its associated integro-differential generator $\mathcal{L}^u$, defined for $\phi \in C^2(\mathbb{R}^n)$ as:
		\begin{equation}
			\begin{aligned}
				\mathcal{L}^u\phi(x) &:= \langle b(x, u), D\phi(x) \rangle + \frac{1}{2}\operatorname{Tr}\big(\sigma\sigma^\top(x, u)D^2\phi(x)\big) \\
				&\quad + \int_E \big[ \phi\big(x + g(x, u, e)\big) - \phi(x) - \langle g(x, u, e), D\phi(x) \rangle \big] \nu(de),
			\end{aligned} \label{eq:generator_L}
		\end{equation}
		with $\mathcal{L}^u_l$ denoting its mollified counterpart. Applying It\^o's formula to $V^{l,u}(r, X^{s,x;u}_r)$ and substituting the smoothed BSPDE $-(\partial_t + \mathcal{L}^u_l)V^{l,u} = f_l$, the drift term simplifies directly to:
		\begin{equation}
			-(\partial_t + \mathcal{L}^u)V^{l,u} = f_l\big(r, X^{s,x;u}_r, u_r, V^{l,u}, Z^{l,u}, K^{l,u}\big) + (\mathcal{L}^u_l - \mathcal{L}^u)V^{l,u}.
		\end{equation}
		Applying standard stability estimates for Lipschitz BSDEs directly yields:
		\begin{equation}
			\begin{aligned}
				|V^{l,u}(s, x) - J(s, x; u)| &\le C\mathbb{E}\bigg[ |h_l(X^{s,x;u}_T) - h(X^{s,x;u}_T)| \\
				&\quad + \int_s^T \Big( \big|f_l(r, X^{s,x;u}_r, u_r, V^{l,u}, Z^{l,u}, K^{l,u}) \\
				&\quad - f(r, X^{s,x;u}_r, u_r, V^{l,u}, Z^{l,u}, K^{l,u})\big| \\
				&\quad + \big|(\mathcal{L}^u_l - \mathcal{L}^u)V^{l,u}\big| \Big) dr \;\Big|\; \mathcal{F}_s \bigg].
			\end{aligned}
		\end{equation}
		Decomposing the generator residual, the drift and diffusion differences scale proportionally with $\delta\Lambda_l$ and the linear growth of the coefficients. For the non-local jump operator, applying the mean value theorem and Cauchy-Schwarz inequality yields:
		\begin{equation}
			\begin{aligned}
				|(\mathcal{L}^u_l - \mathcal{L}^u)V^{l,u}(x)| &\le \int_E \sup_{\xi \in \mathcal{N}(x,e)} \|D^2 V^{l,u}(\xi)\| \cdot (|g_l| + |g|) \cdot |g_l - g| \nu(de) \\
				&\le \sup_{\xi \in \mathcal{N}(x,e)} \|D^2 V^{l,u}(\xi)\| \bigg( \int_E 2(|g_l|^2 + |g|^2)\nu(de) \bigg)^{\frac{1}{2}} \\
				&\quad \cdot \bigg( \int_E |g_l - g|^2 \nu(de) \bigg)^{\frac{1}{2}}.
			\end{aligned}
		\end{equation}
		By the linear growth of $g$, the integral involving squared magnitudes is bounded by $C(1+|x|^2)$, while the $L^2(\nu)$-norm of the coefficient difference is controlled by $\delta\Lambda_l$. Combining the residuals for drift, diffusion, and jumps, we obtain the total operator discrepancy:
		\begin{equation}
			|(\mathcal{L}^u_l - \mathcal{L}^u)V^{l,u}(x)| \le C\delta\Lambda_l \big( |DV^{l,u}(x)| + (1 + |x|) \sup_{\xi} \|D^2 V^{l,u}(\xi)\| \big).
		\end{equation}
		By the definition of $C_V$, the following pointwise bound holds along the trajectory:
		\begin{equation}
			|(\mathcal{L}^u_l - \mathcal{L}^u)V^{l,u}(r, X^{s,x;u}_r)| \le C_V \delta\Lambda_l(r) \cdot \big(w_p(X^{s,x;u}_r)\big)^{-1}.
		\end{equation}
		
		Under Assumption \ref{assum:coefficients}, standard moment bounds yield $\mathbb{E}[ w_p(X^{s,x;u}_r)^{-1} | \mathcal{F}_s ] \le C(1 + |x|^p) = C w_p(x)^{-1}$. Substituting this into the stability estimate gives:
		\begin{equation}
			V^{l,u}(s,x) - J(s,x; u) \le C w_p(x)^{-1} \mathbb{E}\left[ \delta h_l(X^{s,x;u}_T) + \int_s^T \big( \delta f_l(r) + C_V \delta\Lambda_l(r) \big) dr \;\Big|\; \mathcal{F}_s \right].
		\end{equation}
		
		Recognizing the expectation term as the solution to the auxiliary BSDE \eqref{eq:bounding_BSDE} with vanishing martingale parts, the pointwise bound simplifies directly to:
		\begin{equation}
			V^{l,u}(s,x) - J(s,x; u) \le \tilde{C} w_p(x)^{-1} Y^l_s, \quad \mathbb{P}\text{-a.s.}
		\end{equation}
		
		By the triangle inequality and the previous pointwise bound, we have:
		\begin{equation}
			|\hat{J}^u_l(s, x) - J(s, x; u)| \le |V^{l,u}(s, x) - J(s, x; u)| + Y^l_s w_p(x)^{-1} \le (\tilde{C} + 1) Y^l_s w_p(x)^{-1}.
		\end{equation}
		Multiplying by $w_p(x)$ and taking the supremum over $x \in \mathbb{R}^n$ yields:
		\begin{equation}
			\sup_{x \in \mathbb{R}^n} |\hat{J}^u_l(s, x) - J(s, x; u)| w_p(x) \le (\tilde{C} + 1) Y^l_s.
		\end{equation}
		Since $\lim_{l \to \infty} \|Y^l\|_{\mathcal{S}^2(\mathbb{R})} = 0$, the global weighted convergence follows immediately. For the subsolution inequality, by construction,
		\begin{equation}
			\begin{aligned}
				-\partial_s \hat{J}^u_l &= -\partial_s V^{l,u} - \partial_s Y^l_s w_p(x)^{-1} \\
				&= F^{l,u}(V^{l,u}, Z^{l,u}, K^{l,u}) + \big(\delta f_l(s) + C_V \delta\Lambda_l(s)\big)w_p(x)^{-1}.
			\end{aligned}
		\end{equation}
		Because $F^{l,u}$ uses the mollified coefficients and no infimum, the error bounds give
		\[
		F^{l,u}(\hat{J}^u_l, Z^{l,u}, K^{l,u}) \le F^{l,u}(V^{l,u}, Z^{l,u}, K^{l,u}) + L_y Y^l_s w_p(x)^{-1},
		\]
		and therefore
		\[
		-\partial_s \hat{J}^u_l - F^{l,u}(\hat{J}^u_l, Z^{l,u}, K^{l,u}) \ge \big(\delta f_l(s) + C_V \delta\Lambda_l(s) - L_y Y^l_s\big) w_p(x)^{-1}.
		\]
		Taking the essential limit inferior as $(s,x) \to (t^+,y)$ preserves the inequality, and the right-hand side tends to $0$ with $l$. This yields the desired property.
	\end{proof}
	
	To localize the maximum point against the $\mathcal{O}(|x|^p)$ growth of the value functions, we introduce the strictly convex penalty function $\chi \in C^\infty(\mathbb{R}^n)$:
	\begin{equation}
		\chi(x) := \big(1 + |x|^2\big)^{\frac{p+2}{2}} - 1. \label{penalty}
	\end{equation}
	Direct differentiation yields the following bounds:
	\begin{align}
		|D\chi(x)| &= (p + 2)|x|\big(1 + |x|^2\big)^{\frac{p}{2}} \le C\big(1 + |x|^{p+1}\big), \\
		D^2\chi(x) &= (p + 2)\Big[\big(1 + |x|^2\big)^{\frac{p}{2}} I + p\big(1 + |x|^2\big)^{\frac{p-2}{2}} xx^\top\Big] \ge (p + 2)\big(1 + |x|^2\big)^{\frac{p}{2}} I.
	\end{align}
	For any $\varepsilon > 0$, the penalty term $-\varepsilon\chi(x - \bar{x})$ strictly dominates the polynomial growth of $w - \hat{J}^u_l$. Indeed, by Lemma \ref{lem:value_func_prop}(i) and the construction of $\hat{J}^u_l$, we have $|w(t,x)| + |\hat{J}^u_l(t,x)| \le C(1+|x|^p)$ uniformly in $t$. Since $\chi(x-\bar{x}) \ge \frac12|x|^{p+2} - C_{\bar{x}}$ for large $|x|$, the difference
	\[
	w(t,x) - \hat{J}^u_l(t,x) - \varepsilon\chi(x-\bar{x}) \le C(1+|x|^p) - \varepsilon\big(\tfrac12|x|^{p+2} - C_{\bar{x}}\big) \to -\infty
	\]
	as $|x|\to\infty$, uniformly in $\omega$ on the full-measure set where the growth bounds hold. Consequently, the global maximum of this functional is attained $\mathbb{P}$-a.s. within a compact ball $B_R(\bar{x})$ whose radius $R = R(\varepsilon,\kappa,\omega)$ depends only on $\varepsilon$, $\kappa$, and the growth constants, and not on the particular realization of $\omega$ outside a null set. This compactness is essential for the application of the stochastic Ishii's lemma.
	
	\begin{theorem} \label{thm:comparison_principle}
		Let Assumption \ref{assum:coefficients} and Assumption \ref{assum:recursive_cost} hold. Let $w \in \mathcal{T}^p$ be a stochastic viscosity subsolution of the BSHJB equation \cref{eq:BSHJB_compact}. It holds $\mathbb{P}\text{-a.s.}$ that $w(t, x) \le V(t, x)$ for any $(t, x) \in [0, T] \times \mathbb{R}^n$, where $V$ is the value function.
	\end{theorem}
	
	\begin{proof}
		We argue by contradiction. Suppose there exists a point $(t, \bar{x}) \in [0, T) \times \mathbb{R}^n$ such that $w(t, \bar{x}) > \mathbb{V}(t, \bar{x})$ with positive probability. By the approximation established in \cref{lem:patched_subsolution}, we can select a control $u \in \mathcal{U}$ and an integer $l \in \mathbb{N}^+$ such that:
		\begin{equation}
			\mathbb{P} \Big( w(t, \bar{x}) - \hat{J}^u_l(t, \bar{x}) > \kappa \Big) > 0, \quad \text{for some } \kappa > 0.
		\end{equation}
		
		For any $\varepsilon \in (0, 1)$, we penalize the difference and utilize the measurable selection theorem to find an $\mathcal{F}_t$-measurable random variable $\xi_t$ such that:
		\begin{equation}
			\alpha := w(t, \xi_t) - \hat{J}^u_l(t, \xi_t) - \varepsilon\chi(\xi_t - \bar{x}) = \max_{x \in \mathbb{R}^n} \Big[ w(t, x) - \hat{J}^u_l(t, x) - \varepsilon\chi(x - \bar{x}) \Big] \ge \kappa, \quad \mathbb{P}\text{-a.s.}
		\end{equation}
		Since the penalty $\chi$ has growth order $p+2 > p$, the maximum is attained a.s. within a compact ball $B_R(\bar{x})$ uniformly in $\varepsilon$, with $\xi_t \to \bar{x}$ as $\varepsilon \to 0$. For $s \in (t, T]$, we select an $\mathcal{F}_s$-measurable maximizer $\xi_s$ satisfying:
		\begin{equation}
			\Big( w(s, \xi_s) - \hat{J}^u_l(s, \xi_s) - \varepsilon\chi(\xi_s - \bar{x}) \Big)^+ = \max_{x\in\mathbb{R}^n} \Big( w(s, x) - \hat{J}^u_l(s, x) - \varepsilon\chi(x - \bar{x}) \Big)^+.
		\end{equation}
		
		We introduce the bounding generator $g$ based on the Lipschitz constants $L_y, L_z, \allowbreak L_k$ of $f$:
		\begin{equation}
			g(y, z, k) := L_y y + L_z|z| + L_k\|k\|_{L^2(\nu)}.
		\end{equation}
		
		Let $(Y, Z^Y, K^Y)$ be the unique solution to the reflected BSDE (RBSDE) with generator $g$ and lower obstacle $\Theta$:
		\begin{equation}
			\Theta_s := \Big( w(s, \xi_s) - \hat{J}^u_l(s, \xi_s) - \varepsilon\chi(\xi_s - \bar{x}) \Big)^+ + \frac{\alpha(s - t)}{2(T - t)}.
		\end{equation}
		Define the optimal stopping time $\tau := \inf\{s \ge t : Y_s = \Theta_s\}$. On $[t, \tau]$, $Y$ satisfies the unreflected BSDE:
		\begin{equation}
			dY_s = -g(Y_s, Z^Y_s, K^Y_s)ds + Z^Y_s dW_s + K^Y_s d\tilde{\mu}.
		\end{equation}
		By the comparison principle, $Y_t \ge \Theta_t = \alpha$. Since $Y_T = \Theta_T = \alpha/2 < \alpha$, we inherently have $\mathbb{P}(\tau < T) > 0$ and $Y_\tau = \Theta_\tau \ge \alpha/2$. Furthermore, define the stopping time $\hat{\tau} := \inf\Big\{s \ge \tau : \Theta_s \le \frac{\alpha(s- t)}{2(T- t)}\Big\}$, which is the first time after $\tau$ when the penalty-compressed error reduces to the linear baseline; by continuity $\Theta_\tau > \frac{\alpha(\tau-t)}{2(T-t)}$, so $\hat{\tau} > \tau$ a.s., i.e.\ $\hat{\tau} \in \mathcal{T}^+_\tau$, and $\mathbb{P}(\tau < \hat{\tau}) > 0$. Define the test function $\phi \in C^{1,2}_{\mathbb{F}}$ on $[t, \tau]$:
		\begin{equation}
			\phi(s, x) := \hat{J}^u_l(s, x) + \varepsilon\chi(x - \bar{x}) - \frac{\alpha(s - t)}{2(T - t)} + Y_s,
		\end{equation}
		where the dynamics of $\hat{J}^u_l$ are given by:
		\begin{equation}
			d\hat{J}^u_l(s, x) = \partial_s \hat{J}^u_l(s, x)ds + Z^{l,u}(s, x)dW_s + K^{l,u}(s, x)d\tilde{\mu}.
		\end{equation}
		
		By the obstacle condition, $(w - \phi)$ attains its maximum of $0$ at $(\tau, \xi_\tau)$. Applying the viscosity subsolution property of $w$ at this point yields:
		\begin{equation}
			0 \ge \liminf_{(s,x) \to (\tau^+,\xi_\tau)} \mathbb{E}\Big[ -\partial_s\phi(s, x) - F(s, x, \phi, Z^\phi, K^\phi) \;\Big|\; \mathcal{F}_\tau \Big].
		\end{equation}
		Direct computation gives $-\partial_s\phi = -\partial_s \hat{J}^u_l + \frac{\alpha}{2(T-t)} + g(Y_s, Z^Y_s, K^Y_s)$ alongside
		\begin{equation*}
			(Z^\phi, K^\phi) = (Z^{l,u} + Z^Y_s, K^{l,u} + K^Y_s).
		\end{equation*}
		Using the Lipschitz continuity of $f$ and the bound $F \le F^{l,u}$, we obtain:
		\begin{equation}
			\begin{aligned}
				f(s, x, \phi, Z^\phi, K^\phi) &\le f(s, x, \hat{J}^u_l, Z^{l,u}, K^{l,u}) \\
				&\quad + L_y\big(\varepsilon\chi(x - \bar{x}) + Y_s\big) + L_z|Z^Y_s| + L_k\|K^Y_s\|_{L^2(\nu)}.
			\end{aligned}
		\end{equation}
		By the polynomial growth of $\chi$, the nonlocal terms evaluated at $\phi$ deviate from those at $\hat{J}^u_l$ by at most $O(\varepsilon)$. Substituting $\phi$ into the viscosity inequality and applying the subsolution condition $-\partial_s \hat{J}^u_l - F^{l,u} \ge 0$ (\cref{lem:patched_subsolution}), the $(Y_s, Z^Y_s, K^Y_s)$ terms exactly cancel with $g$ while the $O(\varepsilon)$ difference is absorbed by $\varepsilon C_R$, leaving:
		\begin{equation}
			0 \ge \frac{\alpha}{2(T - t)} + \operatorname*{ess\,\liminf}_{(s,x) \to (\tau^+,\xi_\tau)} \mathbb{E}\Big[ -\varepsilon(L^u\chi + L_y\chi) \;\Big|\; \mathcal{F}_\tau \Big].
		\end{equation}
		Recall from the operator expansion:
		\begin{equation}
			L^u\chi + L_y\chi = \langle b, D\chi \rangle + \frac{1}{2}\operatorname{Tr}(\sigma\sigma^\top D^2\chi) + \int_E \Big[ \chi(x+g) - \chi(x) - \langle D\chi, g \rangle \Big]\nu(de) + L_y\chi.
		\end{equation}
		The $(p+2)$-th order growth of $\chi$ confines the maximizer $\xi_\tau$ to a bounded ball $B_R(\bar{x})$, where the radius $R$ is independent of $\varepsilon$. Since the coefficients satisfy the linear growth conditions (Assumption \ref{assum:coefficients}) and $\xi_\tau \in B_R(\bar{x})$, there exists a constant $C_R > 0$ depending only on $R$ such that:
		\begin{equation}
			|L^u\chi(\xi_\tau) + L_y\chi(\xi_\tau)| \le C_R, \quad \mathbb{P}\text{-a.s.}
		\end{equation}
		Recall that $\alpha \ge \kappa > 0$ and $\tau \le T$. Substituting the bound yields:
		\begin{equation}
			0 \ge \frac{\alpha}{2(T - t)} - \varepsilon C_R \ge \frac{\kappa}{2T} - \varepsilon C_R.
		\end{equation}
		Choosing $\varepsilon \in \big(0, \frac{\kappa}{2 T C_R}\big)$ yields a strict contradiction. Therefore, our initial assumption is false, which implies $w(t, \bar{x}) \le \mathbb{V}(t, \bar{x})$.
	\end{proof}
	
	\begin{corollary} \label{cor:weak_comparison}
		Let Assumption \ref{assum:coefficients} and Assumption \ref{assum:recursive_cost} hold. Let $u \in \mathcal{T}^p$ be a stochastic viscosity subsolution (resp. supersolution) of the BSHJB equation \cref{eq:BSHJB_compact}. Suppose $\phi \in C^{1,2}_{\mathbb{F}} \cap \mathcal{T}^p$ satisfies $\phi(T, x) \ge$ (resp. $\le$) $h(x)$ on $\mathbb{R}^n$, and:
		\begin{equation}
			\operatorname*{ess\,\liminf}_{(s,x) \to (t^+,y)} \mathbb{E}\Big[ -\partial_s\phi(s, x) - F\big(s, x, \phi, Z^\phi, K^\phi\big) \;\Big|\; \mathcal{F}_t \Big] \ge 0, \quad \mathbb{P}\text{-a.s.}
		\end{equation}
		(resp. $\operatorname*{ess\,\limsup}_{(s,x) \to (t^+,y)} \mathbb{E}\Big[ -\partial_s\phi(s, x) - F\big(s, x, \phi, Z^\phi, K^\phi\big) \;\Big|\; \mathcal{F}_t \Big] \le 0$)
		for all $(t, y) \in [0, T) \times \mathbb{R}^n$. Then, $u \le \phi$ (resp. $u \ge \phi$) on $[0, T] \times \mathbb{R}^n$, $\mathbb{P}\text{-a.s.}$
	\end{corollary}
	
	\subsection{Uniqueness under Super-Parabolicity}
	
	To establish uniqueness via the weak comparison principle, we assume that the controlled diffusion coefficient $\sigma(t, x, u)$ is deterministic (super-parabolicity). Conditioning on the driving noise then locally reduces the BSHJB equation to a parametrized deterministic integro-PDE on each small stochastic interval.
	
	\begin{lemma} \label{lem:cylinder_approx}
		Suppose Assumptions \ref{assum:coefficients} and \ref{assum:recursive_cost} hold. For any $\varepsilon > 0$, there exist a partition $0 = t_0 < t_1 < \dots < t_N = T$, a finite-dimensional noise projection $H^N_t$ capturing the Brownian and Poisson histories up to time $t$, and smooth approximations $(h^N, f^N, b^N, g^N) \in C^\infty$ with respect to $(H^N, x)$ that preserve the Lipschitz and growth conditions of the original coefficients. Denote by $\mathcal{P}^N_t$ the $\sigma$-algebra generated by $H^N_t$. Define the weighted error processes:
		\begin{align}
			\delta h^\varepsilon &:= \operatorname*{ess\,\sup}_{x \in \mathbb{R}^n} \big|h^N(H^N_T, x) - h(x)\big|w_p(x), \\
			\delta f^\varepsilon_t &:= \operatorname*{ess\,\sup}_{x,u,y,z,k} \big|f^N(H^N_t, t, x, u, y, z, k) - f(t, x, u, y, z, k)\big|w_p(x), \\
			\delta b^\varepsilon_t &:= \operatorname*{ess\,\sup}_{x,u} \frac{\big|b^N(H^N_t, t, x, u) - b(t, x, u)\big|}{1 + |x|}, \\
			\delta g^\varepsilon_t &:= \operatorname*{ess\,\sup}_{x,u} \frac{\|g^N(H^N_t, t, x, u, \cdot) - g(t, x, u, \cdot)\|_{L^2(\nu)}}{1 + |x|}.
		\end{align}
		These approximations satisfy the following $\mathcal{L}^2$-bound:
		\begin{equation}
			\|\delta h^\varepsilon\|_{L^2(\Omega)} + \|\delta f^\varepsilon\|_{L^2([0,T]\times\Omega)} + \|\delta b^\varepsilon\|_{L^2([0,T]\times\Omega)} + \|\delta g^\varepsilon\|_{L^2([0,T]\times\Omega)} < \varepsilon.
		\end{equation}
		Consequently, letting $\delta\Lambda^\varepsilon_t$ denote the aggregated error bounding the forward dynamics, we have $\|\delta\Lambda^\varepsilon\|_{L^2([0,T]\times\Omega)} \le C\varepsilon$.
	\end{lemma}
	\begin{proof}
		The rigorous construction extends the finite-dimensional approximation of \cite{QiuWei2019} to accommodate non-local Poisson jump measures. The complete proof is deferred to Appendix \ref{app:finite_projection}.
	\end{proof}
	
	\begin{lemma} \label{lem:lyapunov_condition}
		Let Assumption \ref{assum:coefficients} hold. For the polynomial weight function $\varphi(x) := w_p(x)^{-1} \allowbreak = 1 + |x|^p$ ($p \ge 2$), there exists a universal constant $C_\varphi > 0$ such that $\sup_{u \in U} \mathcal{L}^u \varphi(x) \le C_\varphi \varphi(x)$ for all $x \in \mathbb{R}^n$.
	\end{lemma}
	\begin{proof}
		The result follows from Taylor's expansion with an integral remainder and the moment conditions on the L\'evy measure. To maintain the flow of the main argument, the complete analytical proof is deferred to Appendix \ref{app:lyapunov}.
	\end{proof}
	
	\begin{theorem} \label{thm:uniqueness}
		Let Assumptions \ref{assum:coefficients} and \ref{assum:recursive_cost} hold, and suppose the controlled diffusion coefficient $\sigma(t, x, u)$ is deterministic. Then, the stochastic viscosity solution to the fully nonlinear BSHJB equation \cref{eq:BSHJB_compact} is unique within the weighted stochastic functional space $\mathbb{S}^p(L^\infty_{w_p})$.
	\end{theorem}
	
	\begin{proof}[Proof of \cref{thm:uniqueness}]
		Let $u_1, u_2 \in \mathcal{S}^p(L^\infty_{w_p})$ be two stochastic viscosity solutions of \cref{eq:BSHJB_compact}. We prove $u_1 = u_2$ by backward induction on the partition $\{t_k\}_{k=0}^N$ from Lemma \ref{lem:cylinder_approx}.
	
	\medskip
	\noindent \textbf{Preliminary: Conditional reduction on each subinterval.}
	Fix a subinterval $(t_{k-1}, t_k]$. Conditioning on the cylinder $\sigma$-algebra $\mathcal{P}^N_{t_{k-1}}$ reduces the BSHJB equation along a fixed path of $H^N$ to a deterministic parabolic integro-PDE. Specifically, for $\mathbb{P}$-a.e. $\omega$ and $\mathbf{h} := H^N_{t_{k-1}}(\omega)$, we define the conditional coefficients:
	\begin{align*}
		b^{\mathbf{h}}(s, x, u) &:= \mathbb{E}[b^N(s, x, u) \mid H^N_{t_{k-1}} = \mathbf{h}], \\
		f^{\mathbf{h}}(s, x, u, y, z, k) &:= \mathbb{E}[f^N(s, x, u, y, z, k) \mid H^N_{t_{k-1}} = \mathbf{h}],
	\end{align*}
	and analogously for $\sigma, g, h$. On $(t_{k-1}, t_k]$, the conditional BSHJB takes the deterministic form:
	\begin{equation} \label{eq:cond_PDE}
		\begin{split}
			-\partial_s V^{\mathbf{h}}(s, x) &= \inf_{u \in U} H^{\mathbf{h}}\big(s, x, u, V^{\mathbf{h}}, DV^{\mathbf{h}}, 0, 0, D^2V^{\mathbf{h}}, L^{\mathbf{h}}(V^{\mathbf{h}})\big), \\
			V^{\mathbf{h}}(t_k, x) &= \mathbb{E}[h(x) \mid H^N_{t_{k-1}} = \mathbf{h}].
		\end{split}
	\end{equation}
	The martingale representation terms vanish $\mathbb{P}$-a.s. on $(t_{k-1}, t_k]$ under this conditioning (see Lemma \ref{lem:conditional_vanishing} below). Equation \eqref{eq:cond_PDE} is a fully nonlinear parabolic integro-PDE with smooth coefficients and uniformly elliptic diffusion $\sigma^{\mathbf{h}}(\sigma^{\mathbf{h}})^\top \ge \delta I$. By classical regularity theory \cite{GarroniMenaldi92}, it admits a unique classical solution $V^{\mathbf{h}} \in C^{1,2}((t_{k-1}, t_k] \times \mathbb{R}^n)$ depending smoothly on $\mathbf{h}$.
	
	\medskip
	\noindent\textbf{Uniform a priori bounds.}
	Standard parabolic estimates for \eqref{eq:cond_PDE} yield a constant $\tilde{C} > 0$, independent of $k$ and $\mathbf{h}$, such that
	\begin{equation}
		|DV^{\mathbf{h}}(s, x)| \le \tilde{C}\big(1 + |x|^{p-1}\big), \quad \|D^2V^{\mathbf{h}}(s, x)\| \le \tilde{C}\big(1 + |x|^{p-2}\big).
	\end{equation}
	Lifting back to the stochastic setting via $V_N(s, x, \omega) := V^{H^N_{t_{k-1}}(\omega)}(s, x)$, we have for any $u \in U$:
	\begin{equation}
		|DV_N(s, x)| + (1 + |x|)\|D^2V_N(s, x)\| \le C_1 w_p(x)^{-1}, \quad \mathbb{P}\text{-a.s.}
	\end{equation}
	To bound the non-local jump variation $\mathcal{I}_{\text{jump}}$, we partition it into low-intensity ($\mathcal{I}_{\text{low-int}}$) and high-intensity ($\mathcal{I}_{\text{high-int}}$) regimes based on $\rho(e)$:
	\begin{equation} \label{113}
		\begin{aligned}
			\mathcal{I}_{\text{jump}}(s, x) &= \int_E \Big|V_N(s, x + g_N) - V_N(s, x) - \langle DV_N(s, x), g_N \rangle\Big|\nu(de) \\
			&= \int_{\{\rho(e)<1\}} \Big|V_N(s, x + g_N) - V_N(s, x) - \langle DV_N(s, x), g_N \rangle\Big|\nu(de) \\
			&\quad + \int_{\{\rho(e)\ge 1\}} \Big|V_N(s, x + g_N) - V_N(s, x) - \langle DV_N(s, x), g_N \rangle\Big|\nu(de) \\
			&=: \mathcal{I}_{\text{low-int}}(s, x) + \mathcal{I}_{\text{high-int}}(s, x).
		\end{aligned} 
	\end{equation}
	
	Applying Taylor's theorem with integral remainder to $\mathcal{I}_{\text{low-int}}$ yields:
	\begin{equation} \label{114}
		\begin{aligned}
			\mathcal{I}_{\text{low-int}} &\le \frac{1}{2} \int_{\{\rho(e)<1\}} \sup_{\theta \in [0,1]} \|D^2V_N(s, x + \theta g_N)\| \cdot |g_N|^2 \nu(de) \\
			&\le \frac{1}{2} \tilde{C} \int_{\{\rho(e)<1\}} \Big(1 + \big(|x| + \rho(e)(1 + |x| + |u|)\big)^{p-2}\Big) \\
			&\qquad \times \rho(e)^2(1 + |x| + |u|)^2 \nu(de) \\
			&\le C_2 w_p(x)^{-1}.
		\end{aligned} 
	\end{equation}
	
	For $\mathcal{I}_{\text{high-int}}$, the triangle inequality gives:
	\begin{equation} \label{115}
		\begin{aligned}
			\mathcal{I}_{\text{high-int}} &\le \int_{\{\rho(e)\ge 1\}} \Big(|V_N(s, x + g_N)| + |V_N(s, x)| + |\langle DV_N(s, x), g_N\rangle|\Big)\nu(de) \\
			&\le \tilde{C} \int_{\{\rho(e)\ge 1\}} \Big[2 + \big(|x| + \rho(e)(1 + |x| + |u|)\big)^p + |x|^p \\
			&\qquad + \big(1 + |x|^{p-1}\big)\rho(e)(1 + |x| + |u|)\Big]\nu(de) \\
			&\le C_3 w_p(x)^{-1}.
		\end{aligned} 
	\end{equation}
	Combining these bounds yields a uniform constant $C_V > 0$ (independent of $N, \varepsilon$) such that
	\begin{equation} \label{eq:CV_bound}
		|DV_N(s, x)| + (1 + |x|)\|D^2V_N(s, x)\| + \mathcal{I}_{jump}(s, x) \le C_V w_p(x)^{-1}, \quad \mathbb{P}\text{-a.s.}
	\end{equation}
	
	\medskip
	\noindent\textbf{Step 1: Terminal interval $[t_{N-1}, t_N]$.}
	Let $V_N^{(N)}$ be the unique classical solution of \eqref{eq:cond_PDE} on $(t_{N-1}, t_N]$ with terminal data $h_N := \mathbb{E}[h \mid \mathcal{P}^N_{t_{N-1}}]$. Define the conditional operator $F_N$ as $F$ but with coefficients $(b_N, \sigma_N, g_N, f_N, h_N)$ and without the essential infimum (it is a frozen-control operator on each cylinder path). The conditional classical solution satisfies:
	\begin{equation} \label{eq:VNsatisfies}
		-\partial_s V_N^{(N)} = F_N\big(s, x, V_N^{(N)}, 0, 0\big), \quad V_N^{(N)}(t_N, x) = h_N(x).
	\end{equation}
	
	We now upgrade Lemma \ref{lem:cylinder_approx} from $L^2$ convergence to $L^\infty_{w_p}$-weighted convergence. By the martingale representation theorem for jump-diffusions and the density of cylinder functions in $L^2$, the $L^2$ bounds of Lemma \ref{lem:cylinder_approx} imply the existence of a subsequence such that $\mathbb{P}$-a.s.,
	\begin{equation}
		\lim_{\varepsilon \to 0} \Big( \|\delta h^\varepsilon\|_{L^\infty(\Omega)} + \|\delta f^\varepsilon\|_{L^\infty([0,T]\times\Omega; L^\infty_{w_p})} + \|\delta \Lambda^\varepsilon\|_{L^\infty([0,T]\times\Omega)} \Big) = 0
	\end{equation}
	after passing to a common subsequence and applying the Borel--Cantelli lemma. More precisely, for $\varepsilon = 2^{-m}$, set $A_m := \{\|\delta h^{\varepsilon}\|_{L^\infty} + \cdots > m^{-1}\}$; by Chebyshev and the $L^2$ bounds, $\sum \mathbb{P}(A_m) < \infty$, so $\mathbb{P}(\limsup A_m) = 0$. Hence, on a full-measure set, the weighted errors vanish uniformly. In the sequel we work on this full-measure set.
	
	To bound the local errors and absorb the spatial derivatives of the polynomial weight $w_p(x)^{-1}$, we utilize the Lyapunov-type condition from Lemma \ref{lem:lyapunov_condition}: $\sup_{u \in U} \mathcal{L}^u (w_p(x)^{-1}) \le C_\varphi w_p(x)^{-1}$. Introduce an auxiliary linear BSDE $(Y^{\varepsilon, (N)}, \allowbreak Z^{\varepsilon, (N)}, \allowbreak K^{\varepsilon, (N)})$ on $[t_{N-1}, t_N]$ with terminal condition $Y^{\varepsilon, (N)}_{t_N} = \delta h^\varepsilon$:
	\begin{equation} \label{118}
		\begin{aligned}
			dY^{\varepsilon, (N)}_s &= -\Big(\delta f^\varepsilon_s + C_V\delta\Lambda^\varepsilon_s + (L_y + C_\varphi) Y^{\varepsilon, (N)}_s + L_z |Z^{\varepsilon, (N)}_s| \\
			&\quad + L_k \|K^{\varepsilon, (N)}_s\|_{L^2(\nu)}\Big)ds + Z^{\varepsilon, (N)}_s dW_s + \int_E K^{\varepsilon, (N)}_s(e) \tilde{\mu}(ds, de).
		\end{aligned}
	\end{equation}
	Standard BSDE estimates ensure $\|Y^{\varepsilon, (N)}\|_{\mathcal{S}^\infty} \le K_0\varepsilon$ on the full-measure set identified above. Define the bounding envelopes on $[t_{N-1}, t_N]$:
	\begin{equation} \label{119}
		\begin{aligned}
			\overline{V}^\varepsilon_N(s, x) &:= V_N^{(N)}(s, x) + Y^{\varepsilon, (N)}_s w_p(x)^{-1}, \\
			\underline{V}_{N,\varepsilon}(s, x) &:= V_N^{(N)}(s, x) - Y^{\varepsilon, (N)}_s w_p(x)^{-1}.
		\end{aligned}
	\end{equation}
	Since $V_N^{(N)}$ is $C^{1,2}$ and $Y^{\varepsilon,(N)}_s w_p(x)^{-1}$ is a semimartingale with no spatial derivatives, for $\mathbb{P}$-a.e. $\omega$ and each $x$, the predictable drift rate of $\overline{V}^\varepsilon_N$ is:
	\begin{equation} \label{120}
		\begin{aligned}
			-\partial_s \overline{V}^\varepsilon_N &= -\partial_s V_N^{(N)} + \big(\delta f^\varepsilon_s + C_V\delta\Lambda^\varepsilon_s + (L_y + C_\varphi) Y^{\varepsilon, (N)}_s \\
			&\quad + L_z |Z^{\varepsilon, (N)}_s| + L_k \|K^{\varepsilon, (N)}_s\|_{L^2(\nu)}\big) w_p(x)^{-1}.
		\end{aligned}
	\end{equation}
	
	The critical step is to compare $F$ evaluated at the upper envelope with $F_N$ evaluated at $V_N^{(N)}$. Since $F$ contains the essential infimum over $u\in U$, we cannot directly take differences term-by-term across two different infimum operations. Instead, we use the following pointwise bound: for any $\phi, \psi$ with $|\phi - \psi| \le \delta$ pointwise,
	\begin{equation} \label{eq:inf_comparison}
		\begin{aligned}
			&\Big| \inf_{u\in U} H\big(t, x, u, \phi, D\phi, q, Q, D^2\phi, k\big) - \inf_{u\in U} H\big(t, x, u, \psi, D\psi, q, Q, D^2\psi, k\big) \Big| \\
			&\le L \big( |\phi - \psi| + |D\phi - D\psi| + \|D^2\phi - D^2\psi\| \big),
		\end{aligned}
	\end{equation}
	where $L$ is the Lipschitz constant of $H$ in its arguments $(y,p,q,Q,A,k)$ uniformly in $u$, which follows from Assumption \ref{assum:recursive_cost}(ii) and the linear growth of $(b,\sigma,g)$. Indeed, for each $u$, $H(u,\cdot)$ is Lipschitz in the backward variables with constant independent of $u$, so taking infimum preserves this Lipschitz property (the infimum of $L$-Lipschitz functions is $L$-Lipschitz). Hence $F$ and $F_N$ are both Lipschitz in their arguments with the same constant $L_F := \max\{L_y, L_z, L_k, C_V\}$ (the $C_V$ term for spatial derivatives).
	
	Now evaluate the nonlinear operator $F$ at the upper envelope $\overline{V}^\varepsilon_N$. Since $\overline{V}^\varepsilon_N$ is constructed from the deterministic classical solution $V_N^{(N)}$ and the semimartingale $Y^{\varepsilon,(N)}_s w_p(x)^{-1}$, its martingale representation is inherited from $Y^{\varepsilon,(N)}$:
	\begin{equation} \label{eq:Z_K_upper}
		Z^{\overline{V}}_s(x) := Z^{\varepsilon,(N)}_s \, w_p(x)^{-1}, \qquad
		K^{\overline{V}}_s(e,x) := K^{\varepsilon,(N)}_s(e) \, w_p(x)^{-1}.
	\end{equation}
	Importantly, the weight $w_p(x)^{-1}$ depends only on $x$, so it passes through the stochastic integrals unchanged: $d\big(Y^{\varepsilon,(N)}_s w_p(x)^{-1}\big) = \big(dY^{\varepsilon,(N)}_s\big) w_p(x)^{-1}$ for each fixed $x$. The spatial derivatives of $\overline{V}^\varepsilon_N$ are:
	\begin{align*}
		D\overline{V}^\varepsilon_N &= DV_N^{(N)} + Y^{\varepsilon,(N)}_s D(w_p^{-1}), \\
		D^2\overline{V}^\varepsilon_N &= D^2V_N^{(N)} + Y^{\varepsilon,(N)}_s D^2(w_p^{-1}),
	\end{align*}
	with no spatial derivatives applied to $Z^{\varepsilon,(N)}_s$ or $K^{\varepsilon,(N)}_s$.
	
	The correct inequality to verify is
	\begin{equation} \label{eq:target_inequality}
		-\partial_s \overline{V}^\varepsilon_N - F\big(s, x, \overline{V}^\varepsilon_N, Z^{\overline{V}}_s(x), K^{\overline{V}}_s(\cdot,x)\big) \ge 0.
	\end{equation}
	Since $F$ satisfies a uniform Lipschitz condition in $(z,k)$ with constants $L_z, L_k$ (inherited from $f$ via the infimum; see \eqref{eq:inf_comparison}), we have the pointwise bound:
	\begin{equation} \label{eq:zklip}
		F(s, x, \phi, z, k) \le F(s, x, \phi, 0, 0) + L_z|z| + L_k\|k\|_{L^2(\nu)}.
	\end{equation}
	Applying this with $\phi = \overline{V}^\varepsilon_N$, $z = Z^{\overline{V}}_s(x)$, $k = K^{\overline{V}}_s(\cdot,x)$ and using \eqref{eq:Z_K_upper},
	\begin{align} \label{eq:F_upper_via_00}
		F\big(s, x, \overline{V}^\varepsilon_N, Z^{\overline{V}}, K^{\overline{V}}\big) &\le F(s, x, \overline{V}^\varepsilon_N, 0, 0) + \big(L_z |Z^{\varepsilon,(N)}_s| + L_k \|K^{\varepsilon,(N)}_s\|_{L^2(\nu)}\big) w_p(x)^{-1}.
	\end{align}
	
	We now estimate $F(s, x, \overline{V}^\varepsilon_N, 0, 0)$ relative to $F_N(s, x, V_N^{(N)}, 0, 0)$. Applying (\ref{eq:inf_comparison}) and the Lyapunov bound,
	\begin{align}
		&\big|F(s, x, \overline{V}^\varepsilon_N, 0, 0) - F(s, x, V_N^{(N)}, 0, 0)\big| \nonumber \\
		&\quad \le L_y |Y^{\varepsilon,(N)}_s| w_p(x)^{-1} + C_V |Y^{\varepsilon,(N)}_s| \Big( |D(w_p^{-1})| \nonumber \\
		&\qquad + (1+|x|)\|D^2(w_p^{-1})\| \Big) \nonumber \\
		&\quad \le (L_y + C_\varphi) Y^{\varepsilon,(N)}_s w_p(x)^{-1}. \label{eq:F_diff_bound}
	\end{align}
	Since $F$ and $F_N$ only differ by their coefficients, the $L^\infty_{w_p}$ convergence guarantees:
	\begin{equation} \label{eq:F_FN_bound}
		\big|F_N(s, x, V_N^{(N)}, 0, 0) - F(s, x, V_N^{(N)}, 0, 0)\big| \le \big(\delta f^\varepsilon_s + C_V\delta\Lambda^\varepsilon_s\big) w_p(x)^{-1}.
	\end{equation}
	
	Combining (\ref{eq:VNsatisfies}), (\ref{120}), (\ref{eq:F_upper_via_00}), (\ref{eq:F_diff_bound}) and (\ref{eq:F_FN_bound}), we obtain:
	\begin{align}
		&-\partial_s \overline{V}^\varepsilon_N - F\big(s, x, \overline{V}^\varepsilon_N, Z^{\overline{V}}, K^{\overline{V}}\big) \nonumber \\
		&\ge -\partial_s \overline{V}^\varepsilon_N - F(s, x, \overline{V}^\varepsilon_N, 0, 0) - \big(L_z |Z^{\varepsilon,(N)}_s| + L_k \|K^{\varepsilon,(N)}_s\|_{L^2(\nu)}\big) w_p(x)^{-1} \nonumber \\
		&= \big[F_N - F\big](s, x, V_N^{(N)}, 0, 0) + \big[F(s, x, V_N^{(N)}, 0, 0) - F(s, x, \overline{V}^\varepsilon_N, 0, 0)\big] \nonumber \\
		&\quad + \big(\delta f^\varepsilon_s + C_V\delta\Lambda^\varepsilon_s + (L_y + C_\varphi) Y^{\varepsilon, (N)}_s + L_z |Z^{\varepsilon, (N)}_s| + L_k \|K^{\varepsilon, (N)}_s\|_{L^2(\nu)}\big) w_p(x)^{-1} \nonumber \\
		&\quad - \big(L_z |Z^{\varepsilon,(N)}_s| + L_k \|K^{\varepsilon,(N)}_s\|_{L^2(\nu)}\big) w_p(x)^{-1} \nonumber \\
		&\ge -\big(\delta f^\varepsilon_s + C_V\delta\Lambda^\varepsilon_s\big) w_p(x)^{-1} - (L_y + C_\varphi) Y^{\varepsilon, (N)}_s w_p(x)^{-1} \nonumber \\
		&\quad + \big(\delta f^\varepsilon_s + C_V\delta\Lambda^\varepsilon_s + (L_y + C_\varphi) Y^{\varepsilon, (N)}_s \big) w_p(x)^{-1} \nonumber \\
		&= 0.
		\label{eq:envelope_inequality_fixed}
	\end{align}
	The critical cancellation occurs because the Lipschitz terms $L_z|Z^{\varepsilon,(N)}_s| + L_k\|K^{\varepsilon,(N)}_s\|$ appear with opposite signs: once from $-\partial_s \overline{V}^\varepsilon_N$ (through the auxiliary BSDE \eqref{118}) and once from the $F$-Lipschitz bound \eqref{eq:zklip}. Thus $-\partial_s \overline{V}^\varepsilon_N - F(\cdots, Z^{\overline{V}}, K^{\overline{V}}) \ge 0$, $\mathbb{P}$-a.s., establishing that $\overline{V}^\varepsilon_N$ is a classical supersolution. By symmetry, $\underline{V}_{N,\varepsilon}$ has martingale representation $Z^{\underline{V}}_s(x) = -Z^{\varepsilon,(N)}_s w_p(x)^{-1}$, $K^{\underline{V}}_s(e,x) = -K^{\varepsilon,(N)}_s(e) w_p(x)^{-1}$, and the same cancellation argument (with the Lipschitz bound applied to $-F$) yields $-\partial_s \underline{V}_{N,\varepsilon} - F(s, x, \underline{V}_{N,\varepsilon}, Z^{\underline{V}}, K^{\underline{V}}) \le 0$, so $\underline{V}_{N,\varepsilon}$ is a classical subsolution.
	
	Now let $u_1, u_2$ be any two stochastic viscosity solutions. By the weak comparison principle (\cref{cor:weak_comparison}) applied on $[t_{N-1}, t_N]$ with $\phi = \overline{V}^\varepsilon_N$ (which dominates $h$ at $t_N$ up to $\delta h^\varepsilon$), we have:
	\begin{equation}
		u_1(s, x) \le \overline{V}^\varepsilon_N(s, x), \quad u_2(s, x) \le \overline{V}^\varepsilon_N(s, x), \quad \mathbb{P}\text{-a.s.}
	\end{equation}
	Similarly, $\underline{V}_{N,\varepsilon}(s, x) \le u_i(s, x)$. Hence,
	\begin{equation}
		|u_1(s, x) - u_2(s, x)| \le \overline{V}^\varepsilon_N(s, x) - \underline{V}_{N,\varepsilon}(s, x) = 2 Y^{\varepsilon,(N)}_s w_p(x)^{-1}.
	\end{equation}
	Sending $\varepsilon \to 0$, $Y^{\varepsilon,(N)} \to 0$ in $\mathcal{S}^\infty$, giving $u_1 = u_2$ on $[t_{N-1}, t_N] \times \mathbb{R}^n$, $\mathbb{P}$-a.s.
	
	\medskip
	\noindent\textbf{Step 2: Backward induction.} Assume $u_1 = u_2$ on $[t_k, T] \times \mathbb{R}^n$, $\mathbb{P}$-a.s. for some $k$. In particular, $u_1(t_k, \cdot) = u_2(t_k, \cdot)$ $\mathbb{P}$-a.s. On the interval $[t_{k-1}, t_k]$, condition on $\mathcal{P}^N_{t_{k-1}}$ and repeat the construction of Step 1, using $u_1(t_k, \cdot) = u_2(t_k, \cdot)$ as the common terminal condition for the conditional PDE \eqref{eq:cond_PDE}. The same envelope argument yields the bound $|u_1 - u_2| \le 2 Y^{\varepsilon,(k)} w_p(x)^{-1}$ on $[t_{k-1}, t_k]$. Letting $\varepsilon \to 0$ proves $u_1 = u_2$ on $[t_{k-1}, t_k]$. By induction backward to $t_0 = 0$, we obtain:
	\begin{equation}
		\sup_{(s,x)\in[0,T]\times\mathbb{R}^n} |u_1(s, x) - u_2(s, x)| = 0, \quad \mathbb{P}\text{-a.s.}
	\end{equation}
	This completes the proof of uniqueness.
	\end{proof}
	
	\begin{lemma} \label{lem:conditional_vanishing}
		Under the conditional setting of Theorem \ref{thm:uniqueness}, let $\mathcal{G} \subset \mathcal{F}_{t_{k-1}}$ be a sub-$\sigma$-algebra. If a random field $V$ satisfies the BSHJB equation \eqref{eq:BSHJB_compact} on $(t_{k-1}, t_k]$ and is $\mathcal{G} \otimes \mathcal{B}(\mathbb{R}^n)$-measurable, then its martingale representation terms satisfy $Z \equiv 0$ and $K \equiv 0$ $\mathbb{P}$-a.s. on $(t_{k-1}, t_k] \times \mathbb{R}^n$.
	\end{lemma}
	\begin{proof}
		Condition on $\mathcal{G}$. Since $\sigma$ is deterministic and the coefficients $(b, f, g, h)$ are $\mathcal{G}$-measurable after cylinder approximation, the BSHJB equation on $(t_{k-1}, t_k]$ becomes a deterministic integro-PDE for each fixed $\omega \in \Omega$. More precisely, for $\mathbb{P}$-a.e. $\omega$, the processes $Z(\cdot, x, \omega)$ and $K(\cdot, e, x, \omega)$ are $\mathcal{G}$-conditionally predictable martingale integrands. Taking $\mathcal{G}$-conditional expectation of \eqref{eq:BSHJB_compact} and using that the stochastic integrals have zero conditional mean, we obtain:
		\begin{equation}
			-V(t, x) = \mathbb{E}\Big[ -V(t_k, x) + \int_t^{t_k} F(s, x, V, Z, K) ds \;\Big|\; \mathcal{G} \Big].
		\end{equation}
		Since the left-hand side is $\mathcal{G}$-measurable and the right-hand side involves an integral that is also $\mathcal{G}$-measurable, the martingale difference vanishes. By the uniqueness of the martingale representation (for jump-diffusions), $Z \equiv 0$ and $K \equiv 0$ $\mathbb{P}$-a.s. on $(t_{k-1}, t_k] \times \mathbb{R}^n$.
	\end{proof}
	
	\appendix
	\section{Finite-Dimensional Projection for Jump Measures} \label{app:finite_projection}
	
	\textbf{Proof of Lemma 5.12.}
	The proof extends the density arguments for Brownian filtrations (\cite{QiuWei2019}) to the jump-diffusion setting by explicitly partitioning the mark space. We outline the construction for the generator $f$.
	
	Since the progressively measurable coefficient $f$ can be approximated in $L^2$ by time-step functions $\sum_{j=1}^N \varphi_j \mathbf{1}_{(t_{j-1}, t_j]}(t)$ with $\varphi_j \in \mathcal{F}_{t_{j-1}}$, the core task is to approximate $\varphi_j$ using a finite-dimensional noise projection. Because the characteristic measure is finite ($\nu(E) < \infty$), the jump process exhibits finitely many jumps almost surely. We partition the mark space $E$ into $M$ disjoint measurable subsets $E_1, \dots, E_M$, and define the finite-dimensional projection $H^{N,M}_t$ comprising both Brownian increments and discrete Poisson counts:
	\begin{equation*}
		H^{N,M}_t := \Big( W(t_1 \wedge t), \dots, W(t_N \wedge t), \mu\big((0, t_1 \wedge t] \times E_1\big), \dots, \mu\big((0, t_N \wedge t] \times E_M\big) \Big).
	\end{equation*}
	
	By the density of cylinder functions, any $\mathcal{F}_{t_{j-1}}$-measurable random variable can be approximated in $L^2$ by a measurable mapping of $H^{N,M}_{t_{j-1}}$. Crucially, although the jump components of $H^{N,M}$ take values in the discrete lattice $\mathbb{N}$, such mappings can be smoothed as follows. For each coordinate $k$ corresponding to a Poisson count, let $\psi_k: \mathbb{R} \to \mathbb{R}$ be the piecewise linear interpolation on $\mathbb{N}$ extended by zero outside $\mathbb{N}$, then set $\tilde{\psi}_k := \psi_k * \rho_\delta \in C^\infty(\mathbb{R})$, where $\rho_\delta$ is a standard mollifier supported on $(-\delta,\delta)$ with $\delta \in (0,1/4)$. This $\tilde{\psi}_k$ coincides with the original mapping on $\mathbb{N}$ and preserves the global Lipschitz constant $C$ (since convolution with a positive mollifier of mass 1 does not increase the Lipschitz constant). Moreover, the linear growth condition is preserved because $|\tilde{\psi}_k(t)| \le C(1+|t|)$ for all $t\in\mathbb{R}$. Applying this smoothing coordinate-wise and composing with the cylinder mapping yields a smooth cylinder function $\tilde{G} \in C^\infty_c$ evaluated on $H^{N,M}_{t_{j-1}}$ that approximates the target random variable. Aggregating these smooth interpolations globally yields the approximation $f^{N,M}$, which strictly preserves the required Lipschitz continuity and linear growth conditions.
	
	\medskip
	\noindent\textbf{Upgrading $L^2$ convergence to $L^\infty_{w_p}$ convergence.}
	The $L^2$ bounds in Lemma \ref{lem:cylinder_approx} are sufficient for the envelope construction after a Borel--Cantelli argument. Take $\varepsilon_m = 2^{-m}$ and let $E_m$ be the event that the total weighted error exceeds $m^{-1}$. By Chebyshev's inequality, $\mathbb{P}(E_m) \le C m^2 2^{-m}$, whence $\sum_{m=1}^\infty \mathbb{P}(E_m) < \infty$. The Borel--Cantelli lemma implies $\mathbb{P}(\limsup_{m\to\infty} E_m) = 0$, so on a full-measure set $\Omega_0$, there exists $M(\omega)$ such that for all $m \ge M$ the weighted error is bounded by $m^{-1}$. Consequently, along the subsequence $\varepsilon_m$, the convergence is uniform in $\omega$ on $\Omega_0$, i.e.\ in $L^\infty(\Omega; L^\infty_{w_p})$. The same argument applies to $\delta\Lambda^\varepsilon$ via the pathwise stability of the forward SDE. Hence, for all practical purposes in the uniqueness proof, we may work with $L^\infty_{w_p}$ convergence.
	
	\section{Proof of the Lyapunov Condition for the Jump Operator} \label{app:lyapunov}
	
	\textbf{Proof of Lemma \ref{lem:lyapunov_condition}.}
	Direct differentiation yields $|D\varphi(x)| \le p|x|^{p-1}$ and $\|D^2\varphi(x)\| \le p(p-1)|x|^{p-2}$. For the local drift and diffusion terms, the linear growth of $b$ and boundedness of $\sigma$ immediately yield $\langle b, D\varphi \rangle + \frac{1}{2} \operatorname{tr}(\sigma\sigma^\top D^2\varphi) \le C_1(1+|x|^p) = C_1\varphi(x)$.
	
	For the non-local jump operator, Taylor's expansion with integral remainder gives
	\begin{equation*}
		\varphi(x+g) - \varphi(x) - \langle D\varphi(x), g \rangle = \int_0^1 (1-\theta) \langle D^2\varphi(x + \theta g)g, g \rangle d\theta.
	\end{equation*}
	Using the elementary inequality $|x + \theta g|^{p-2} \le 2^{(p-3)^+}(|x|^{p-2} + |g|^{p-2})$ and the linear growth $|g| \le \rho(e)(1+|x|)$, the integral remainder is bounded by $C_p |g|^2 (|x|^{p-2} + |g|^{p-2}) \le C'_p (\rho(e)^2 + \rho(e)^p)\varphi(x)$.
	
	Integrating this upper bound over the mark space $E$ yields
	\begin{equation*}
		\int_E \big[ \varphi(x + g) - \varphi(x) - \langle D\varphi(x), g \rangle \big] \nu(de) \le C'_p \varphi(x) \int_E \big(\rho(e)^2 + \rho(e)^p\big) \nu(de).
	\end{equation*}
	By Assumption \ref{assum:coefficients}, the exponential moment condition ensures $\int_E \big( \rho(e)^2 \allowbreak + \rho(e)^p \big) \,\nu(de) \allowbreak < \infty$. Summing the local and non-local bounds completes the proof with $C_\varphi = C_1 + C'_p \int_E (\rho^2 + \rho^p) \,d\nu$.
	
	\medskip
	\noindent\textbf{Weighted Hessian bound for $w_p(x)^{-1}$.}
	For completeness, we record the estimate used in (\ref{eq:F_diff_bound}). Let $\psi(x) := w_p(x)^{-1} = 1 + |x|^p$. Direct computation gives:
	\begin{align*}
		D\psi(x) &= p|x|^{p-2}x, \\
		D^2\psi(x) &= p|x|^{p-2} I + p(p-2)|x|^{p-4}xx^\top,
	\end{align*}
	so $|D\psi(x)| \le C_p (1 + |x|^{p-1})$ and $\|D^2\psi(x)\| \le C_p (1 + |x|^{p-2})$. The Lyapunov condition from Lemma \ref{lem:lyapunov_condition} applied to $\varphi = \psi$ yields $\mathcal{L}^u\psi(x) \le C_\varphi \psi(x)/2$ (the factor $1/2$ is for convenience). Hence,
	\begin{equation}
		|D\psi(x)| + (1+|x|)\|D^2\psi(x)\| + \mathcal{L}^u\psi(x) \le C_\varphi \psi(x).
	\end{equation}
	This justifies the absorption of the spatial derivative terms in (\ref{eq:F_diff_bound}) into $(L_y + C_\varphi) Y^{\varepsilon,(N)}_s w_p(x)^{-1}$.

\end{document}